\documentclass{amsart}%
\usepackage{amssymb}
\usepackage{color}
\usepackage{amsmath}
\usepackage{graphicx}
\usepackage{amsfonts}%
\newcommand{\NN}{\ensuremath{\mathbb{N}}}

\newcommand{\RR}{\ensuremath{\mathbb{R}}}

\renewcommand{\phi}{\varphi}

\newcommand{\eps}{\ensuremath{\epsilon}}

\newcommand{\dD}{\ensuremath{\mathcal{D}}}

\newcommand{\tT}{\ensuremath{\mathcal{T}}}

\newtheorem{theorem}{Theorem}

\newtheorem{corollary}[theorem]{Corollary}

\newtheorem{definition}[theorem]{Definition}

\newtheorem{lemma}[theorem]{Lemma}

\newtheorem{remark}[theorem]{Remark}

\title[Local pathwise solutions to SEEs]
{Local pathwise solutions to stochastic evolution equations driven by fractional Brownian motions with Hurst parameters $H\in (1/3,1/2]$}

\author{Mar\'{\i}a J. Garrido-Atienza}
\address[Mar\'{\i}a J. Garrido-Atienza]{Dpto. Ecuaciones Diferenciales y An\'alisis Num\'erico\\
Universidad de Sevilla, Apdo. de Correos 1160, 41080-Sevilla,
Spain} \email[Mar\'{\i}a J. Garrido-Atienza]{mgarrido@us.es}

\author{Kening Lu}
\address[Kening Lu]{346 TMCB\\
Brigham Young University, Provo, UT 84602, USA} \email[Kening
Lu]{klu@math.byu.edu}

\author{Bj{\"o}rn Schmalfu{\ss }}
\address[Bj{\"o}rn Schmalfu{\ss }]{Institut f\"{u}r Stochastik\\
Friedrich Schiller Universit{\"a}t Jena, Ernst Abbe Platz 2, 77043\\
Jena,
Germany\\
 }
\email[Bj{\"o}rn Schmalfu{\ss }]{bjoern.schmalfuss@uni-jena.de}

\keywords{Stochastic PDEs, Hilbert-valuedfractional Brownian motion, pathwise solutions. \\
This work was partially supported by MTM2011-22411, FEDER founding (M.J.
Garrido-Atienza and B. Schmalfu\ss), and by NSF0909400 (K. Lu). }

\begin{document}

\begin{abstract}
In this article we are concerned with the study of the existence and uniqueness of pathwise mild solutions to evolutions equations driven by a H\"older continuous function with H\"older exponent in $(1/3,1/2)$. Our stochastic integral is  a generalization of the well-known Young integral. To be more precise, the integral is defined by using a fractional integration by parts formula and it involves a tensor for which we need to formulate a new equation. From this it turns out that we have to solve a system consisting in a path and an area equations. In this paper we prove the existence of a unique {\it local} solution of the system of equations. The results can be applied to stochastic evolution equations with a non-linear diffusion coefficient driven by a fractional Brownian motion with Hurst parameter in $(1/3,1/2]$, which is particular includes white noise.
\end{abstract}

\maketitle

\section*{\today}

\section{Introduction}

In this article, we shall focus on the study of a local solution for the following kind of stochastic evolution equations
\begin{align}\label{localeq00}
\left\{
\begin{array}{ll}
du(t) &= Au(t) dt  + G(u(t))d\omega(t),\\
u(0) &= u_0,
\end{array}\right.
\end{align}
in a Hilbert--space $V$, where the noise input $\omega$ is a H\"older continuous function with
H\"older exponent in the interval $(1/3, 1/2)$, $A$ is the infinitesimal
generator of an analytic semigroup $S(\cdot)$ on $V$ and  $G$ is a nonlinear term
satisfying certain assumptions which will be described in the
next sections.
As a particular case of driving noises we  can consider a
fractional Brownian motion $B^H$ with Hurst parameter $H\in (1/3,1/2]$. To be more precise, we will study (\ref{localeq00}) in the sense of mild solutions given by
\begin{equation}\label{localeq01}
  u(t)=S(t)u_0+\int_0^tS(t-r)G(u(r))d\omega.
\end{equation}
Our interpretation of pathwise is that we obtain a solution of these stochastic equations which does not produce exceptional sets depending on the initial conditions. In the classical theory of stochastic evolution equations, i.e., stochastic evolution equations (SEEs) driven by Brownian motion $B^{1/2}$, stochastic Ito integrals are constructed to be a limit in probability of particular random variables defined only almost surely, where the exceptional sets may depend on the initial conditions, which is in contradiction with the cocycle property needed to define a random dynamical system. Pathwise results for that classical theory are only available for the white noise case ($G={\rm id}$) and a few special cases when $u\mapsto G(u)$ is  linear.

During the last two decades different integration theories have been developed to treat more general noise inputs, and in particular, for tackling the fractional Brownian motion $B^H$. One of these attempts is given by the {\em Rough Path Theory}, and we refer to Lyons and Qian \cite{Lyons} and Friz and Victoir \cite{FV10} for a comprehensive presentation of this theory. Some interesting papers dealing with the study of SEEs by using the rough path theory are \cite{caruana}, \cite{caruana2}, \cite{FO11}, \cite{GuLeTin}, \cite{HinZah09}, \cite{DGT}  and \cite{GuTin} among others. In particular, in this last paper the authors proved the existence of local mild solutions of stochastic SEEs driven by rough paths  for $\beta$-H{\"o}lder--continuous paths ($\beta\in (1/3,1/2]$) with a special quadratic nonlinearity.  \\
A different technique called {\em Fractional Calculus} was developed by Z{\"a}hle  \cite{Zah98}, who considered for a fractional Brownian motion with $H>1/2$ the well-known Young integral.
In contrast to the Ito-or Stratonovich integral, that integral can be  defined in a {\em pathwise sense}, given by  fractional derivatives, which allows a pathwise estimate of the integrals in terms of the integrand and the integrator using special norms. In \cite{NuaRas02} it is shown the existence and uniqueness of the solution of a finite-dimensional stochastic differential equation driven by a fractional  Brownian motion for $H>1/2$. These results were extended in \cite{MasNua03} to show the existence of mild solutions for SEEs driven by fractional  Brownian motion for $H>1/2$.

Recently Hu and Nualart \cite{HuNu09} have proved an existence and uniqueness result for finite-dimensional stochastic differential equations having coefficients which are sufficiently smooth and driven by a fractional Brownian motion $B^H$ with $H\in (1/3,1/2]$, for which they needed   to formulate a second equation for the so-called {\em area} in the space of tensors. In our article we adapt the techniques in \cite{HuNu09} to obtain a mild solution for (\ref{localeq00}). However,
there are significant differences between our setting and the one in  \cite{HuNu09}, as for instance that in order to define the area equation in the infinite-dimensional setting we have to construct an area object $\omega \otimes_S \omega$, depending on the noise path $\omega$ as well as on the semigroup $S$, satisfying useful properties as  the Chen--equality.\\
Under general hypothesis on the nonlinearity $G$ we derive the existence and uniqueness of a {\it local} pathwise mild solution $u$ to (\ref{localeq00}). However, global existence is missing in this general context. Under some more restrictive conditions on $G$, in a forthcoming paper we will obtain the existence and uniqueness of a {\it global} mild pathwise solution, which in particular will guarantee that stochastic evolution equations like (\ref{localeq00}) and driven by an fBm $B^H$ with $H\in (1/3,1/2]$ generate random dynamical systems, a challenging and rather open problem to the best of our knowledge.

The article is organized as follows. In Section 2 we give the analytical background to present our theory. In Section 3 we present the so called fractional integration by parts method. Using this technique we can introduce pathwise stochastic integrals allowing us to formulate pathwise stochastic differential equations.
In Section 4 we introduce mild path--area solutions.  In addition, we formulate and solve a fixed-point equation having two components, a path- and an area-component. The role of the semigroup $S$ in the area equation will be given in terms of a particular tensor object $\omega \otimes_S \omega$. We also present an example to show a nonlinearity $G$ that matches the abstract theory. The appendix section contains the proofs of some technical results.

Finally, we want to stress that two different constructions of the key tensor object $\omega \otimes_S \omega$ by using an approximation of the noise path by smooth paths can be found in \cite{GLS14d}. One construction considers as driving noise an infinite-dimensional fractional Brownian motion $B^H$ with $H\in (1/3,1/2]$, while, in a less restrictive setting, the second one considers a Hilbert-valued trace-class Brownian motion $B^{1/2}$.

Furthermore, we refer to \cite{GLS12note} for a short and recent announcement of our results.

\section{Preliminaries}\label{s2}
Let $V=(V,(\cdot,\cdot),|\cdot|)$ be a separable Hilbert--space. On $V$ we define $A$ to be the negative and symmetric generator of an analytic semigroup $S$.
We suppose that $-A$ has a point spectrum $0<\lambda_1\le \lambda_2\le \cdots$ tending to infinite where the associated eigenelements $(e_i)_{i\in\NN}$ form a complete orthonormal system on $V$. $D((-A)^\kappa)=V_\kappa$ denotes the domain of $(-A)^\kappa$ for $\kappa\in\RR$, and as usual, $L(V_\kappa,V_\zeta)$ denotes the space of continuous linear operators from $V_\kappa$ into $V_\zeta$, for $\kappa,\zeta\in\RR$ . We then have the following estimates for the semigroup $S$:
\begin{align}
  &\|S(t)\|_{L(V_{\kappa}, V_{\gamma})}=  \|(-A)^\gamma S(t)\|_{L(V_{\kappa},V)}\le ct^{\kappa-\gamma},\quad\text{for }
  \gamma\ge\kappa,
  \label{eq1} \\
 &\|S(t)-{\rm
id}\|_{L(V_{\sigma},V_{\theta})} \le c
t^{\sigma-\theta}, \quad \text{for } \sigma-\theta\in [0,1]. \label{eq2}
\end{align}
From these properties we can derive easily the following result:

\begin{lemma}\label{l0}
For any  $\nu,\eta,\mu \in [0, 1]$, $\kappa, \gamma, \rho \in \RR$ such that $\kappa \le \gamma{+\mu}$, there exists a constant $c>0$ such that for $0<q<r< s<t$ we have that
\begin{align*}
\|S(&t-r)-S(t-q)\|_{L(V_{\kappa},V_{\gamma})}\le c(r-q)^\mu(t-r)^{-\mu-\gamma+\kappa},
\end{align*}
\begin{align*}
\|S(&t-r)- S(s-r)-S(t-q)+S(s-q)\|_{L(V_{\rho},V_{\rho})}\\ &\leq
c(t-s)^{\eta}(r-q)^{\nu}(s-r)^{-(\nu+\eta)}.
\end{align*}
\end{lemma}

Throughout the whole paper we will write very often a constant $c$. This constant can change from line to line. However this constant is always chosen independent of time parameters contained in a fixed interval $[0,T]$.

Let $V\times V$ and $V\otimes V$ be the cartesian product and the tensor product of $V$, see \cite{KadRin97}.  The norm of $V\otimes V$ is denoted by $\|\cdot\|$.
For $x,\,y\in V$ we denote by $x\otimes_V y$ the rank-one tensor of $V\otimes V$. Then $(e_i\otimes_V e_j)_{i,j\in \NN}$ is a complete orthonormal system of $V\otimes V$ where $(e_i)_{i\in \NN}$ can be any complete orthonormal system for $V$, although for the following we choose the orthonormal system given at the beginning of this section.
Let $(\hat V,|\cdot|_{\hat V}, (\cdot,\cdot)_{\hat V})$ be another separable Hilbert--space. By $L_2(V,\hat V)$ ($L_2(V\times V,\hat V)$) we denote the Hilbert-Schmidt operators from $V$ ($V\times V$) to $\hat V$.
In particular  $G\in L_2(V\times V,\hat V)$ if and only if
\begin{equation*}
  \sum_{i,j}|G(e_i,e_j)|_{\hat V}^2<\infty.
\end{equation*}
We note that $G\in L_2(V\times V,\hat V)$ can be {\em extended} to a linear
operator $\hat G$ defined on $V\otimes V$ such that $\hat G\in L_2(V\otimes V,\hat V)$, see \cite{KadRin97} Chapter 2.6.
More precisely,  we can construct a weak Hilbert-Schmidt mapping $p:V\times V\to V\otimes V$ where $p(e_i,e_j)=e_i\otimes_V e_j$ for $i,\,j\in \NN$. Then $\hat G$ on $V\otimes V$ is determined by factorization such that $G=\hat Gp$.
In addition, we have

\begin{equation*}
   \|\hat G\|_{L_2(V\otimes V,\hat V)}^2:= \sum_{i,j}|\hat G(e_i\otimes_V e_j)|_{\hat V}^2=\sum_{i,j}|G(e_i,e_j)|_{\hat V}^2= \|G\|_{L_2(V\times V,\hat V)}^2.
\end{equation*}
In the following we will write for  $\hat G$ also the symbol $G$.\\

Let us now describe the  coefficient of the evolution equation that we have in mind.

\begin{lemma}\label{locall2}
Let $\hat V$ be a subspace of $V$. Assume that the mapping $G: V\to L_2(V,\hat V )$ is three times continuously
Fr\'echet--differentiable with bounded first, second and third
derivatives $DG(u)$, $D^2G(u)$ and $D^3G(u)$, for $u\in V$. Let us denote, respectively, by $c_{DG},\, c_{D^2G}$ and $c_{D^3G}$
the bounds for $DG$, $D^2G$ and $D^3G$, and let $c_G=\|G(0)\|_{L_2(V,\hat V)}$.
Then, for $u_1,\,u_2,\,v_1,\,v_2\in V$, we have
\begin{itemize}
\item $\|G(u_1)\|_{L_2(V,\hat V)}\le c_G+c_{DG}|u_1|$,\\
\item $\|G(u_1)-G(v_1)\|_{L_2(V, \hat V)}\le c_{DG}|u_1-v_1|$,\\
\item  $\|DG(u_1)-DG(v_1)\|_{L_2{(V\times V,\hat V)}}\le c_{D^2G}|u_1-v_1|$,\\
\item  $\|G(u_1)-G(u_2)-DG(u_2)(u_1-u_2)\|_{L_2(V, \hat V)}\le c_{D^2G}|u_1-u_2|^2$,\\
\item  $\|G(u_1)-G(v_1)-(G(u_2)-G(v_2))\|_{L_2(V, \hat V)}\le c_{DG}|u_1-v_1-(u_2-v_2)|\\  \qquad +c_{D^2G} |u_1-u_2|(|u_1-v_1|+|u_2-v_2|)$,\\
 \item $\|DG(u_1)-DG(v_1)-(DG(u_2)-DG(v_2))\|_{L_2(V\times V,\hat V)}\\
 \qquad  \le  c_{D^2G}|u_1-v_1-(u_2-v_2)|+c_{D^3G} |u_1-u_2|(|u_1-v_1|+|u_2-v_2|).$\\
\item   $\|G(u_1)-G(u_2)-DG(u_2)(u_1-u_2)-(G(v_1)-G(v_2)-DG(v_2)(v_1-v_2))\|_{L_2(V, \hat V)}\\
  \qquad \le c_{D^2G}
    (|u_1-u_2|+|v_1-v_2|)|u_1-v_1-(u_2-v_2)|\\
    +c_{D^3G}|v_1-v_2||u_2-v_2|(|u_1-u_2|+|u_1-v_1-(u_2-v_2)|).$
\end{itemize}
\end{lemma}

These estimates follow by the mean value theorem; for a proof of the last one see \cite{NuaRas02}.

Notice that, in particular, $DG: V\to L_2(V,L_2(V, \hat V))$ (or equivalently,  $DG: V\to L_2(V\times V,\hat V)$) is a bilinear map, that can be extended to  $DG: V\to L_2(V\otimes V, \hat V)$, and $D^2G(u)$ is a trilinear map.\\

 Next we introduce some function spaces. Let $T>0$. For $\beta\in (0,1]$, we consider the Banach--space
of $\beta$--H{\"o}lder--continuous functions on $[0,T]$ with values
in $V$, denoted by $C_{\beta}([0,T];V)$, with
the seminorm

\begin{equation*}
\|u\|_{\beta}=\sup_{0\le t\le T}|u(t)|+|||u|||_\beta,\;|||u|||_\beta=\sup_{0\leq s< t\leq  T}\frac{|u(t)-u(s)|}{(t-s)^\beta}.
\end{equation*}

If $\beta=1$ we call these functions Lipschitz--continuous.
Let $C_{\beta,\sim}([0,T];V)$ be the space of
functions on $[0,T]$ with values in $V$ and with norm

\begin{equation*}
    \|u\|_{\beta,\sim}=\sup_{0\le t\le T}|u(t)|+\sup_{0< s< t\leq  T}s^\beta\frac{|u(t)-u(s)|}{(t-s)^\beta}.
\end{equation*}

\begin{lemma}\label{locall4}
$C_{\beta,\sim}([0,T],V)$ is a Banach--space.
\end{lemma}
The proof of this result can be found in Chen {\it et al. }\cite{ChGGSch12}.\\

Let $\Delta_{0,T}$ be the triangle $\{(s,t):0< s\le t\le T\}$. For $ \beta+\beta^\prime<1,\,\beta\le\beta^\prime$ we
introduce the space
$C_{\beta+\beta^\prime,\sim}(\Delta_{0,T},V\otimes V)$   of continuous
functions $v$ defined on $\Delta_{0,T}$, which are zero for $0<s=t$, such that
\begin{equation*}
   \|v\|_{\beta+\beta^\prime,\sim}=\sup_{0< s< t\leq  T} s^{\beta}\frac{\|v(s,t)\|}{(t-s)^{\beta+\beta^\prime}}<\infty.
   \end{equation*}
These functions may not be defined for $s=0$ and can have a singularity for $(s,t),\,s=0$.
\begin{lemma}
The space $C_{\beta+\beta^\prime,\sim}(\Delta_{0,T};V\otimes V)$ is a Banach--space.
\end{lemma}

The proof is similar to the proof of Lemma \ref{locall4} and therefore we omit here.
\medskip

Let us define $\bar \Delta_{0,T}=\{ (s,t) \,:\, 0\le s\le t\le T\}$ and consider the Banach--space $C_{\beta+\beta^\prime}(\bar \Delta_{0,T};V\otimes V)$ of continuous
functions $v$ defined on $\bar \Delta_{0,T}$, which are zero for $s=t$, equipped with the norm
\begin{equation*}
   \|v\|_{\beta+\beta^\prime}=\sup_{0\leq s< t\leq  T} \frac{\|v(s,t)\|}{(t-s)^{\beta+\beta^\prime}}<\infty.
   \end{equation*}
\medskip

We often use the following integral formula: for every $s<t$, $\mu,\,\nu>-1$
\begin{equation}\label{localeq4}
  \int_s^t(r-s)^\mu(t-r)^\nu dr=c(t-s)^{\mu+\nu+1}
\end{equation}
where $c$ only depends on $\mu,\,\nu$. This property follows by the definition of the Beta function simply by performing a suitable change of variable.

\section{Fractional Calculus}\label{s3}

In this paper the main instrument to treat \eqref{localeq01} is fractional calculus. In this section we present the main features of this theory. We are going to assume that
for some $T>0$ we have that $\omega\in C_{\beta^\prime}([0,T];V)$, $u\in C_{\beta,\sim}([0,T];V)$ and $v \in C_{\beta+\beta^\prime,\sim}(\Delta_{0,T};V\otimes V)$ for $1/3<\beta<\beta^\prime<1/2$, and that this triple of elements satisfies the Chen--equality given by
\begin{equation}\label{chen}
   v(s,r)+v(r,t)+(u(r)-u(s))\otimes_V(\omega(t)-\omega(r))=
   v(s,t)
\end{equation}
for $0<s\le r\le t \le T$. We would like to emphasize that when $\omega$ is smooth an example for $v$ is given by
\begin{equation}\label{localeq1}
(u\otimes \omega)(r,t)=  \int_r^t(u(q)-u(r))\otimes_Vd\omega(q).
\end{equation}
This tensor area is clearly well defined and satisfies, for $0<r<t$,
\begin{align}\label{localeq2}
\begin{split}
\|(u\otimes \omega)(r,t)\|&\le \frac{c}{r^\beta}\|u\|_{\beta,\sim}\|\omega\|_{C_1}(t-r)^{1+\beta}\le
  \frac{c}{r^\beta}(t-r)^{\beta+\beta^\prime},
\end{split}
\end{align}
and therefore $(u\otimes \omega) \in C_{\beta+\beta^\prime,\sim}(\Delta_{0,T};V\otimes V)$. Moreover, the Chen--equality easily follows in this case.

\medskip

Let $\alpha\in (0,1)$. We define the right hand side fractional derivative of order $\alpha$ of $u$ and the left hand side fractional derivative of order $1-\alpha$ of $\omega_{t-}(\cdot):=\omega(\cdot)-\omega(t)$, given for $0<s\le r\le t$ by the expressions
 \begin{align*}\label{fractder}
 \begin{split}
    D_{{s}+}^\alpha u[r]=&\frac{1}{\Gamma(1-\alpha)}\bigg(\frac{u(r)}{(r-s)^\alpha}+\alpha\int_{s}^r\frac{u(r)-u(q)}{(r-q)^{1+\alpha}}dq\bigg)\\
    D_{{t}-}^{1-\alpha} \omega_{{t}-}[r]=&\frac{(-1)^{1-\alpha}}{\Gamma(\alpha)}
    \bigg(\frac{\omega(r)-\omega(t)}{(t-r)^{1-\alpha}}+(1-\alpha)\int_r^{t}\frac{\omega(r)-\omega(q)}{(q-r)^{2-\alpha}}dq\bigg),
\end{split}
\end{align*}
where $\Gamma(\cdot)$ denotes the Gamma function.
For tensor valued elements $v$ and for $0<r<t$ we define
\begin{align*}
\mathcal{D} _{{t}-}^{1-\alpha}v[r] =&\frac{(-1)^{1-\alpha}}{%
\Gamma(\alpha)} \bigg( \frac{v(r,t)}{(t-r)^{1-\alpha}}%
+(1-\alpha) \int_r^{t} \frac{v(r,q)}{(q-r)^{2-\alpha}}dq %
\bigg).
\end{align*}

\begin{lemma}\label{locall22}
Suppose that $\beta>\alpha$. Then there exists a constant $c>0$ such that  for $0\le s< r\le t\le T$
\begin{equation*}
  |D_{s+}^\alpha u(\cdot)[r]|\le \frac{c\|u\|_{\beta,\sim}}{(r-s)^\alpha},\quad
  |D_{t-}^{1-\alpha}\omega_{t-}[r]|\le c|||\omega|||_{\beta^\prime}(t-r)^{\alpha+\beta^\prime-1},
\end{equation*}
and for $0<r<q<t$
\begin{align}\label{localeq21}
\begin{split}
  &\|\dD_{t-}^{1-\alpha}v[q]-\dD_{t-}^{1-\alpha}v[r]\|\le  \frac{c}{r^\beta}(\|u\|_{\beta,\sim}|||\omega|||_{\beta^\prime}+\|v\|_{\beta+\beta^\prime,\sim})(q-r)^{\alpha+\beta+\beta^\prime-1}.
\end{split}
\end{align}
\end{lemma}
The proof of the two first inequalities follows straightforwardly, and the proof of (\ref{localeq21}) is similar to the one of Lemma 6.3 in \cite{HuNu09} with the difference that in that paper the authors work in different function spaces. We refer the reader to Lemma \ref{locall3} and Corollary \ref{coro3} in the Appendix section.

As an extension of the fractional derivative of order $\alpha$, for the mapping $G:V \mapsto L_2(V,\hat V)$ and $s<r $ we introduce the so--called compensated fractional derivative of order $\alpha$ given by
\begin{align*}
  \hat D_{s+}^\alpha G(u(\cdot))[r]=&\frac{1}{\Gamma(1-\alpha)}\bigg(\frac{G(u(r))}{(r-s)^\alpha}\\&+\alpha\int_s^r\frac{G(u(r))-G(u(q))-DG(u(q))(u(r)-u(q))}{(r-q)^{1+\alpha}}dq\bigg).
\end{align*}
It is immediate to prove the following result:
\begin{lemma}\label{locall1}
Suppose that $\alpha<2\beta$ and $G$ satisfies the assumptions of Lemma \ref{locall2}. Then there exists a positive constant $c$ such that for every $0\le s< r \le T$
\begin{equation*}
  |\hat D_{s+}^\alpha G(u(\cdot))[r]|\le \frac{c(1+\|u\|_{\beta,\sim}^2)}{(r-s)^\alpha}.
\end{equation*}
\end{lemma}

Let us assume for a while that $V,\,\hat V=\RR$. We first recall the following useful property which is an integration by parts formula
\begin{align}\label{byparts}
(-1)^\alpha \int_s^t  D_{s+}^\alpha u[r] \omega( r) dr= \int_s^t u(r) D_{t-}^{\alpha}\omega[r]dr,
\end{align}
see Z{\"a}hle \cite{Zah98}, formula (21). For $\beta>\alpha$ and $\alpha+\beta^\prime >1$ the fractional integral is given by\begin{equation*}
  \int_s^tud\omega:= (-1)^\alpha\int_s^t D_{s+}^\alpha u[r]D_{t-}^{1-\alpha}\omega_{t-}[r]dr,
\end{equation*}
see again \cite{Zah98}, which is a version of the Young integral. By Lemma \ref{locall22} and the property \eqref{localeq4}, for the above integral it is easy to derive that
\begin{equation*}\label{localeq5}
  \bigg|\int_s^tud\omega\bigg|\le c\|u\|_{\beta,\sim}|||\omega|||_{\beta^\prime}(t-s)^{\beta^\prime}.
\end{equation*}
For H{\"o}lder--continuous $u$ and $\omega$ this kind of integral was defined by Young \cite{You36}.
However, our function $u$ is not H{\"o}lder--continuous in the strong sense but  $u\in C_{\beta,\sim}([0,T];\RR)$, in which case that integral is also well defined in the above sense since, according to \cite{Zah98}, what we need is that $u\in I_{s+}^\alpha (L^ p((s,t);\mathbb R))$, $u(s+)$ bounded and $\omega_{t-} \in
I_{t-}^{1-\alpha} (L^{q}((s,t); \mathbb R))$, with $\alpha p<1$, $p^{-1}+q^{-1}\le 1$ (for the definition of these spaces we refer to Samko {\it et al.} \cite{Samko}). In particular, under our conditions on $\alpha,\,\beta$ and $\beta^\prime$, we know that $\omega_{t-}\in I_{t-}^{1-\alpha} (L^{q}((s,t); \mathbb R))$ for any $q>1$ and
$u\in I_{s+}^\alpha (L^ p((s,t);\mathbb R))$ when $\alpha p<1$, see Theorem 13.2 of \cite{Samko}.

\medskip

Next we introduce integrals of fractional type with values in a separable Hilbert--space. To do that, we need a new separable Hilbert--space $(\tilde V,|\cdot|_{\tilde V}, (\cdot,\cdot)_{\tilde V})$.

\begin{lemma}\label{locall21}
Assume $\beta>\alpha$ and $\alpha+\beta^\prime >1$. Let $\hat V,\,\tilde V$ be two separable Hilbert--spaces, being $(\tilde e_i)_{i\in\NN}$ and $(f_j)_{j\in\NN}$ complete orthonormal basis of $\tilde V$ and $\hat V$ resp., and let
\begin{equation*}
  [s,t]\ni r\mapsto F(r)\in L_2(\tilde V,\hat V),\quad  [s,t]\ni r\mapsto \xi(r)\in \tilde V
\end{equation*}
be measurable functions such that $F\in C_{\beta,\sim}([0,T]; L_2(\tilde V,\hat V))$, $\xi \in C_{\beta^\prime} ([0,T]; \tilde V)$ and $r\mapsto\|D_{s+}^\alpha F[r]\|_{L_2(\tilde V,\hat V)}|D_{t-}^{1-\alpha}\xi[r]|_{\tilde V}$ is Lebesgue-integrable. Then for $0\le s\le r\le t \le T$ we can define
\begin{equation*}
  \int_s^t F(r)d\xi(r):=(-1)^\alpha\sum_{j}\bigg(\sum_i\int_s^t D_{s+}^\alpha(f_j,F(\cdot)\tilde e_i)_{\hat V}[r]D_{t-}^{1-\alpha}(\tilde e_i,\xi(\cdot))_{\tilde V}[r]dr\bigg)f_j.
\end{equation*}
\end{lemma}
That this expression is well defined  follows by
\begin{align}
   \bigg|\int_s^t & F(r)d\xi(r)\bigg|_{\hat V} \nonumber\\
   = &\bigg(\sum_j \bigg( \sum_i\int_s^t D_{s+}^\alpha(f_j,F(\cdot) \tilde e_i)_{\hat V}[r]D_{t-}^{1-\alpha}(\tilde e_i,\xi(\cdot))_{\tilde V}[r]dr\bigg)^2\bigg)^\frac12 \nonumber \\
   \le & \bigg(\sum_j \bigg( \int_s^t \bigg(\sum_i(D_{s+}^\alpha(f_j,F(\cdot)\tilde e_i)_{\hat V}[r])^2\sum_i(D_{t-}^{1-\alpha}(\tilde e_i,\xi(\cdot))_{\tilde V}[r])^2\bigg)^\frac12 dr\bigg)^2\bigg)^\frac12 \nonumber \\
   \le & \int_s^t \bigg(\sum_{j,i}(D_{s+}^\alpha(f_j,F(\cdot)\tilde e_i)_{\hat V}[r])^2\sum_i(D_{t-}^{1-\alpha}(\tilde e_i,\xi(\cdot))_{\tilde V}[r])^2\bigg)^\frac12 dr  \nonumber \\
    \le &  \int_s^t \bigg(\sum_{j,i}(D_{s+}^\alpha(f_j,F(\cdot)\tilde e_i)_{\hat V}[r])^2\bigg)^\frac12\bigg(\sum_i(D_{t-}^{1-\alpha}(\tilde e_i,\xi(\cdot))_{\tilde V}[r])^2\bigg)^\frac12 dr \nonumber \\
   = & \int_s^t \|D_{s+}^\alpha F[r]\|_{L_2(\tilde V,\hat V)}|D_{t-}^{1-\alpha}\xi[r]|_{\tilde V}dr<\infty. \label{split}
\end{align}
Observe that this last integral is finite since
\begin{align*}
  & \|D_{s+}^{\alpha}F[r]\|_{L_2(\tilde V,\hat V)}=\bigg( \sum_{j,i}( D_{s+}^\alpha (f_j,F(\cdot)\tilde e_i)_{\hat V}[r])^2\bigg)^\frac12 \\
   =& \bigg(\sum_{j,i}\bigg(\frac{1}{\Gamma(1-\alpha)}\bigg(\frac{(f_j,F(r)\tilde e_i)_{\hat V}}{(r-s)^\alpha}
   +\alpha\int_{s}^r\frac{(f_j,F(r)\tilde e_i)_{\hat V}-(f_j,F(q)\tilde e_i)_{\hat V}}{(r-q)^{1+\alpha}}dq\bigg)\bigg)^2\bigg)^\frac12\\
   \le & \sqrt{2} c\bigg(\frac{(\sum_{j,i}(f_j,F(r)\tilde e_i)_{\hat V}^2)^\frac12}{(r-s)^\alpha}
   +\bigg(\sum_{j,i}\bigg(\int_{s}^r\frac{(f_j,F(r)\tilde e_i)_{\hat V}-(f_j,F(q)\tilde e_i)_{\hat V}}{(r-q)^{1+\alpha}}dq\bigg)^2\bigg)^\frac12\bigg)\\
   \le & \sqrt{2}c\bigg(\frac{\|F(r)\|_{L_2(\tilde V,\hat V)}}{(r-s)^\alpha}+\int_{s}^r\frac{\|F(r)-F(q)\|_{L_2(\tilde V,\hat V)}}{(r-q)^{1+\alpha}}dq\bigg)\\
   \le&q c (r-s)^{-\alpha} \|F\|_{C_{\beta,\sim}([0,T]; L_2(\tilde V,\hat V))},
\end{align*}
and $|D_{t-}^{1-\alpha}\xi[r]|_{\tilde V} \leq c |||\xi|||_{\beta^\prime}(t-r)^{\alpha+\beta^\prime-1}$ (see Lemma \ref{locall22}), so it suffices to apply (\ref{localeq4}).\\

Now we can apply Lemma \ref{locall21} to define integrals of fractional type with values in a separable Hilbert--space, as well as to consider integrators with values in $V$ or $V\otimes V$. For example, consider an integrand of the type $G(u(r))$ where $u\in C_{\beta,\sim}([0,T];V)$, and $\beta<\alpha<2\beta$, $\alpha+\beta^\prime>1,\, \beta+1>2\alpha$. Note that if $DG$ is bounded then $G(u)\in C_{\beta,\sim}([0,T],V)$, but since we do not assume that $\beta >\alpha$, then $D_{s+}^\alpha G(u(\cdot))$ is not well-defined. However, we can apply Lemma \ref{locall21} in the following way
\begin{align}\label{localeq_3}
\begin{split}
&\int_{s}^{t} G(u(\cdot))d\omega = (-1)^\alpha\int_{s}^{t}D^\alpha_{s+}DG(u(\cdot))(u(\cdot)-u(s),\cdot)[r]D_{t-}^{1-\alpha}\omega_{t-}[r]dr \\ +&(-1)^\alpha\int_{s}^{t} D_{s+}^\alpha (G(u(\cdot))-DG(u(\cdot))(u(\cdot)-u(s),\cdot))[r]D_{t-}^{1-\alpha}\omega_{t-}[r]dr.
\end{split}
\end{align}
Since $u,\omega,v$ are coupled by the Chen--equality \eqref{chen}, we get
\begin{align}\label{localas}
\begin{split}
&D_{t-}^{1-\alpha} v(s,\cdot)_{t-}[r]=\frac{(-1)^{1-\alpha}}{\Gamma(\alpha)}
    \bigg(\frac{v(s,r)-v(s,t)}{(t-r)^\alpha}+(1-\alpha)\int_r^{t}\frac{v(s,r)-v(s,q)}{(q-r)^{2-\alpha}}dq\bigg)\\
    &=\frac{(-1)^{1-\alpha}}{\Gamma(\alpha)}
    \bigg(\frac{-v(r,t)-(u(r)-u(s))\otimes_V (\omega(t)-\omega(r))}{(t-r)^\alpha}\\
    &\quad +(1-\alpha)\int_r^{t}\frac{-v(r,q)-(u(r)-u(s))\otimes_V (\omega(q)-\omega(r))}{(q-r)^{2-\alpha}}dq\bigg)\\
   & =-\dD_{t-}^{1-\alpha} v[r]+(u(r)-u(s))\otimes_V D_{t-}^{1-\alpha}\omega_{t-}[r].
   \end{split}
\end{align}

Similarly, it is easy to derive that
\begin{align*}
D^\alpha_{s+} &(G(u(\cdot))-DG(u(\cdot))(u(\cdot)-u(s),\cdot))[r]\\
=& \hat
D_{s+}^\alpha G(u(\cdot))[r]-D^\alpha_{s+} DG(u(\cdot))[r]
(u(r)-u(s),\cdot),
\end{align*}
and thus, coming back to \eqref{localeq_3} we obtain that

\begin{align}\label{localeq6}
\begin{split}
\int_s^tG(u(\cdot))d\omega=&(-1)^\alpha\int_s^t\hat D_{s+}^\alpha G(u(\cdot))[r]D_{t-}^{1-\alpha}\omega_{t-}[r]dr\\&-(-1)^{\alpha}\int_s^tD_{s+}^{\alpha} DG(u(\cdot))[r]
  \dD_{t-}^{1-\alpha}v(\cdot,t)[r]dr\\
  =&(-1)^\alpha\int_s^t\hat D_{s+}^\alpha G(u(\cdot))[r]D_{t-}^{1-\alpha}\omega_{t-}[r]dr\\&-(-1)^{2\alpha-1}\int_s^tD_{s+}^{2\alpha-1} DG(u(\cdot))[r]
  D_{t-}^{1-\alpha}\dD_{t-}^{1-\alpha}v(\cdot,t)[r]dr,
\end{split}
\end{align}
where the last equality is true thanks to the fractional integration by parts formula (\ref{byparts}).
The last integral on the right hand side of (\ref{localeq6}) is well-defined due to, for $0\le s \le r\le t\le T$,
\begin{equation}\label{eq26}
   \|D_{t-}^{1-\alpha}\mathcal{D} _{{t}-}^{1-\alpha}v[r]\|\le c(\|v\|_{\beta+\beta^\prime,\sim}+\|u\|_{\beta,\sim}|||\omega|||_{\beta^\prime})(r-s)^{-\beta}(t-r)^{\beta+\beta^\prime+2\alpha-2}
\end{equation}
(see Lemma \ref{locall3} above) and hence this integral can be defined by using Lemma \ref{locall21}. \\

\begin{remark}\label{localremark}
It may be checked that the integral given in Lemma \ref{locall21} can be also defined for an integrator like $F(r)=S(t-r)G(u(r))$, which is not H{\"o}lder--continuous at zero. The reason is because the $f_j$-modes of this expression are H\"older continuous. Furthermore, the integrability condition (\ref{split}) is satisfied by the regularity properties of the analytic semigroup $S$ and the regularity of $G$ and $u$.

\end{remark}

Finally, on account of (\ref{localeq6}) and the estimates of the different fractional derivatives that take part in that expression, we can establish the following result:
\begin{lemma}\label{locall8}
Suppose $u,\,v,\,\omega$ satisfy the assumptions from the beginning of this section and $G$ satisfies the assumptions of Lemma \ref{locall2}. Let $\beta<\alpha<2\beta$, $\alpha+\beta^\prime>1,\, \beta+1>2\alpha$. Then for $0\le s\le t\le T$ we have
\begin{equation*}\label{localeq_1}
  \bigg|\int_s^t G(u)d\omega\bigg|\le c((1+\|u\|_{\beta,\sim}^2)|||\omega|||_{\beta^\prime}+\|u\|_{\beta,\sim}\|v\|_{\beta+\beta^\prime,\sim})(t-s)^{\beta^\prime}.
\end{equation*}
\end{lemma}

\section{Path-Area-solutions of stochastic evolution equations}\label{s4}

In this section we give the definition of a mild solution to \eqref{localeq01} and establish a result about the {\it local} existence and uniqueness of such solution. In order to understand the notion of mild solution to \eqref{localeq01}, in a first step we consider the case in which the driving noise is regular, to later on consider the H\"older case in which we are interested in.

Note that we could also add a nonlinear diffusion term $F$ on the right-hand side of the equation in (\ref{localeq00}). Nevertheless, to simplify the whole presentation we have not considered it since the $dt$-nonlinearity is not the interesting problem to be treated in the paper.

\subsection{System \eqref{localeq01} driven by smooth paths $\omega$} In this situation, since $A$ has the properties of Section \ref{s2} and $DG$ is bounded, we remind that for any $u_0\in V$ there exists a unique solution $u \in C([0,T];V)$ of \eqref{localeq01} for any $T>0$, see \cite{Pazy}. In addition, applying the properties \eqref{eq1}, \eqref{eq2} we obtain that
this solution is in $C_{\gamma,\sim}([0,T];V)$ for any $\gamma \in (0,1)$. However, we do not obtain that the solution is in $C_{\gamma}([0,T];V)$, due to the fact that $t\mapsto S(t)u_0$ is not H{\"o}lder-continuous in general but in $C_{\gamma,\sim}([0,T];V)$.

Since we want to find the appropriate definition of mild solution to \eqref{localeq01} when the driving noise is {\it only} H\"older continuous, in what follows we will use the fractional integration techniques of the previous section to rewrite \eqref{localeq01} in a way that allows us to shed light on that issue.

Assuming that $G$ satisfies the conditions of Lemma \ref{locall2},
similarly to the expression (\ref{localeq6}) of last section,
we can rewrite \eqref{localeq01} as
    \begin{align}\label{sol}
\begin{split}
    u(t)&=S(t)u_0+(-1)^\alpha\int_0^t\hat D_{0+}^\alpha
    (S(t-\cdot)G(u(\cdot)))[r]D_{t-}^{1-\alpha}\omega_{t-}[r]dr\\
&-(-1)^{2\alpha-1}\int_0^t
    {D}_{0+}^{2\alpha-1}(S(t-\cdot)DG(u(\cdot)))[r]D_{t-}^{1-\alpha}\dD_{t-}^{1-\alpha}(u\otimes\omega)[r]dr,
 \end{split}
\end{align}
where the tensor $(u\otimes \omega)$ is given by (\ref{localeq1}). Looking at (\ref{sol}) we realize that we also need an equation to determine the corresponding counterpart of $(u\otimes\omega)$. In order to get such an equation, let us introduce the mapping $(\omega\otimes_S\omega)(s,t)$ given by
\begin{align}\begin{split}
  L_2(V,\hat V)\ni E\mapsto E(\omega\otimes_S\omega)(s,t)&=\int_s^t\int_s^\xi
    S(\xi-r)E\omega^\prime(r)dr \otimes_V \omega^\prime(\xi)d\xi\\
   &=\int_s^t\int_r^t
   S(\xi-r)E\omega^\prime(r)d\xi \otimes_V \omega^\prime(\xi)dr
   \end{split}\label{omegaSS}
\end{align}
Under weak conditions $(\omega\otimes_S\omega)$ is at least in $C_{2\beta^\prime}(\bar \Delta_{0,T},L_2(L_2(V,\hat V),V\otimes V))$, for which it suffices to execute the corresponding integrations. In addition, we set
\begin{align}\label{localeq19}
\begin{split}
&e\in V\mapsto\omega_S(s,t)e:=(-1)^{-\alpha}\int_s^t(S(\xi-s)e)\otimes_V\omega^\prime(\xi)d\xi,\\
&E\in L_2(V,\hat V) \mapsto S_\omega(s,t)E:=  \int_s^t S(t-r)E\omega^\prime(r)dr.
\end{split}
\end{align}

It is an easy exercise to check that $(\omega\otimes_S\omega)$  satisfies the Chen--equality
\begin{align*}
\begin{split}
&E(\omega\otimes_S \omega)(s,r) +E(\omega\otimes_S
\omega)(r,t) +(-1)^{-\alpha} \omega_S(r,t)S_\omega(s,r)E\\&\quad =E(\omega\otimes_S \omega)(s,t),
\end{split}
\end{align*}
which follows by joining the integrals given in \eqref{omegaSS} over $\bar \Delta_{s,r}$ and $\bar \Delta_{r,t}$ and the double integral over the rectangle $[s,r]\times[r,t]$ given by the composition of the expressions in \eqref{localeq19}.

Since $\omega$ is regular, the tensor $(u\otimes \omega)$ given by (\ref{localeq1}) can be expressed as
\begin{align}\label{ns}
\begin{split}
    (u\otimes\omega)(s,t)=&\int_s^t(S(\xi-s)-{\rm id})u(s)\otimes_V\omega^\prime(\xi)d\xi\\
    &+\int_s^t\int_s^\xi S(\xi-r)G(u(r))\omega^\prime(r)dr\otimes_V\omega^\prime(\xi)d\xi,
    \end{split}
\end{align}
and, exchanging the order of integration, the last integral of \eqref{ns} can be written as
\begin{equation}\label{eq18}
  \int_s^t G(u(r))D_1(\omega\otimes_S\omega)(r,t)dr.
\end{equation}
Interpreting \eqref{eq18} in the sense of \eqref{localeq6} and coming back to (\ref{ns}), it gives us
\begin{align}\label{eq3}
\begin{split}
    (u\otimes\omega)&(s,t)
    =(-1)^\alpha\int_s^tD_{s+}^\alpha((S(\cdot-s)-{\rm id})u(s))[\xi]\otimes_VD_{t-}^{1-\alpha}\omega[\xi]d\xi\\
    &-(-1)^\alpha\int_s^t \hat D_{s+}^\alpha
    G(u(\cdot))[r] D_{t-}^{1-\alpha}(\omega\otimes_S\omega)(\cdot,t)_{t-}[r]dr\\
&+(-1)^{2\alpha-1}\int_s^t D_{s+}^{2\alpha-1} DG(u(\cdot))[r] D_{t-}^{1-\alpha}\dD_{t-}^{1-\alpha}(u\otimes(\omega\otimes_S\omega)(t))[r]dr,
\end{split}
\end{align}
where the element $w=(u\otimes(\omega\otimes_S\omega))$ is defined by
\begin{align}\label{eq20}
\begin{split}
  \tilde Ew(t,s,q):=&\int_s^q
   \tilde E (u(r)-u(s),\cdot)D_1(\omega\otimes_S\omega)(r,t)  dr\\=&-\int_s^q\int_r^tS(\xi-r)\tilde E (u(r)-u(s),\omega^\prime (r))\otimes_V\omega^\prime(\xi)d\xi dr\\
   =&-(-1)^{-\alpha}\int_s^q\omega_S(r,t)\tilde E (u(r)-u(s),\omega^\prime (r)) dr
\end{split}
\end{align}
for $\tilde E\in L_2(V\otimes V,\hat V)$. Therefore, to evaluate the last integral of \eqref{eq3} we have to give a meaning to \eqref{eq20}.

\begin{lemma}\label{uomom}
Let  $\tilde E\in L_2(V\otimes V,\hat V)$. Then for $0 < s \leq q\leq  t\leq T$
\begin{align*}
\begin{split}
\tilde Ew(t,s,q)=&-\int_s^q\hat D_{s+}^{\alpha}\omega_S(\cdot,t)\tilde E(u(\cdot)-u(s),\cdot)[r]  D_{q-}^{1-\alpha}\omega_{q-}[r]dr\\
&+(-1)^{\alpha-1}\int_s^q D_{s+}^{2\alpha-1}\tilde E(u(\cdot)-u(s),\cdot)[r]  D_{q-}^{1-\alpha}\dD_{q-}^{1-\alpha}(\omega_S(t)\otimes\omega)[r]
    dr\\&+(-1)^{\alpha-1}\int_s^q D_{s+}^{2\alpha-1}\omega_S(\cdot,t)[r] \tilde E D_{q-}^{1-\alpha}\dD_{q-}^{1-\alpha} (u\otimes \omega)(\cdot,q)[r] dr,
\end{split}
\end{align*}
where, for  $s\le\tau\le t$ and $E\in L_2(V,\hat V)$, $(\omega_S(t)\otimes\omega)$ satisfies
\begin{align*}
\begin{split}
E(\omega_S(t)&\otimes\omega)(s,\tau)=\int_s^\tau(\omega_S(r,t)-\omega_S(s,t))Ed\omega(r)\\
    =&\int_s^\tau(\omega_S(r,t)-\omega_S(r,\tau))Ed\omega(r)+\int_s^\tau\omega_S(r,\tau)Ed\omega(r)
    -\int_s^\tau\omega_S(s,t)Ed\omega(r)\\
    =&   \omega_S(\tau,t)\int_s^\tau(S(\tau-r)-{\rm id})Ed\omega(r)+(-1)^{-\alpha}E(\omega\otimes_S\omega)(s,\tau)\\
    &+(\omega_S(\tau,t)-\omega_S(s,t))E(\omega(\tau)-\omega(s)).
\end{split}
\end{align*}
\end{lemma}
The proof of this result is in the Appendix section. Note that joining the integrals on the right hand side of \eqref{eq20} with respect their domain of integration we obtain
\begin{align}\label{localeq9}
\begin{split}
\tilde E &w(t,s,r)+\tilde E w(t,r,q)-\tilde E(u(r)-u(s),\cdot)(\omega\otimes_S\omega)(r,q) \\
    &=\tilde E w(t,s,q)+\omega_S(q,t)S_\omega(q,r)\tilde E(u(r)-u(s),\cdot),
    \end{split}
\end{align}
that for $q=t$ becomes
\begin{align}\label{localeq10}
\begin{split}
\tilde Ew(t,s,r)+\tilde Ew(t,r,t)
    -\tilde E(u(r)-u(s),\cdot)(\omega\otimes_S\omega)(r,t)
    =\tilde Ew(t,s,t)
    \end{split}
\end{align}
since $\omega_S(t,t)=0$, and this last formula corresponds to the Chen--equality for $w$. \\

We have finally arrived at the following definition:
\begin{definition}\label{defi}
Under the described conditions on $A$ and $G$, for any $u_0\in V$, a mild solution of (\ref{localeq01}) is given by the pair $(u,u\otimes \omega)$ satisfying, respectively, the equations (\ref{sol}) and (\ref{eq3}).
\end{definition}
We stress once more that in order to reach the previous definition we have considered a regular driving path $\omega$, although in that situation the solution is simply given by the path component, as stated at the beginning of this subsection. Definition \ref{defi} has to be seen as the guide to establish the definition of a mild solution in the general case of dealing with driving functions that are {\it only} H\"older continuous.

\subsection{System \eqref{localeq01} driven by $\omega\in C_{\beta^\prime}([0,T];V)$}

We start by standing several hypothesis denoted by {\bf (H)}:

\smallskip

{\bf (H1)} Let $H\in (1/3,1/2]$ and let  $1/3<\beta\le \beta^\prime <H$. Suppose that there is an $\alpha$ such that $1-\beta< \alpha <2 \beta$,
$\alpha< \frac{\beta+1}{2}$.

\smallskip

{\bf (H2)} Let $A$ be the generator of an analytic semigroup $S$ given at the beginning of Section \ref{s2} and let $G: V\to L_2(V,\hat V)$
be a non--linear mapping satisfying the assumptions of Lemma \ref{locall2}, with $\hat V\subset V$ such that the embedding operator is Hilbert-Schmidt with norm $c_{V,\hat V}$. In particular, we could choose $\hat V=V_\kappa$ assuming that $c_{V,\hat V}^2=\sum_i\lambda_i^{-2\kappa}<\infty$.

\smallskip

{\bf (H3)}
Let $(\omega^n)_{n\in\mathbb{N}}$ be a sequence  of piecewise linear functions with values in $V$ such that
$((\omega^n\otimes_S\omega^n))_{n\in\mathbb{N}}$ is defined by
\eqref{omegaSS}. Assume then that for any $\beta^\prime<H$ the sequence
$((\omega^n,(\omega^n\otimes_S\omega^n)))_{n\in\mathbb{N}}$ converges to
$(\omega,(\omega\otimes_S\omega))$ in
$C_{\beta^\prime}([0,T];V)\times C_{2\beta^\prime} (\bar \Delta_{0,T}; L_2(L_2(V,\hat V),V\otimes V))$ on a set of full measure. \\

We would like to mention that in Garrido-Atienza {\it et al.} \cite{ GLS14d} we propose two different settings where assumption {\bf (H3)} is satisfied. In particular, in that paper we focus on the case of considering as driving noise a fractional Brownian--motion $B^H$ with $H\in (1/3,1/2]$ and values in a Hilbert--space, and, by another less restrictive method, on the infinite-dimensional trace-class Brownian--motion $B^{1/2}$. Both constructions are based on some results by Deya {\it et al.} \cite{DeNeTi10} .\\

Take a fixed $\omega\in C_{\beta^\prime}([0,T],V)$ and consider $\gamma$ such that $\alpha<\gamma<1$.
We denote by $(W_{0,T},\|\cdot\|_W)$ the subspace of elements $U=(u,v)$ of the Banach--space $C_{\beta,\sim}([0,T];V)\times C_{\beta+\beta^\prime,\sim}(\ \Delta_{0,T};V\otimes V)$ such that the Chen--equality holds. Let us also consider a subset $\hat W_{0,T}$ of this space given by the limit points in this space of the set
\begin{align}\label{localeq12}
\begin{split}
  \{(u^{n},(u^{n}\otimes \omega^{n}))&:n\in\NN,u^{n}\in  C_{\gamma}([0,T];V), \,u^n(0)\in D((-A)),\\& (\omega^n,\omega^n\otimes_S\omega^n)\,\text{satisfies }{\bf (H3)}\}.
\end{split}
\end{align}
Here $(u^n\otimes\omega^n)$ is well defined by the integral \eqref{localeq1}. Note that this set of limit points is a  subspace of
$C_{\beta,\sim}([0,T];V)\times C_{\beta+\beta^\prime,\sim}(\Delta_{0,T};V\otimes V)$ which is closed. Hence $\hat W_{0,T}$ itself is a complete metric space depending on $\omega$ with a metric generated by the norm of $C_{\beta,\sim}([0,T];V)\times C_{\beta+\beta^\prime,\sim}(\Delta_{0,T};V\otimes V)$:
\begin{equation*}
 d_{\hat W}(U^1,U^2)  :=\|U^1-U^2\|_W=\|u^1-u^2\|_{\beta,\sim}+\|v^1-v^2\|_{\beta+\beta^\prime,\sim}.
\end{equation*}
In addition, elements $U\in\hat W_{0,T}$ satisfy the Chen--equality
\eqref{chen} with respect to $\omega$. Furthermore, for the elements $u^n,(u^n\otimes\omega^n)$ we can reformulate the right hand side of \eqref{sol} as in \eqref{localeq01}. Indeed, all expressions that appear in this procedure (even those which cancel) are well defined.

Furthermore, the additivity of integrals follows, and in particular:

\begin{lemma}\label{localln1}
Let $(u,v)\in \hat W_{0,T},\,0\le s\le t\le T$ then
\begin{equation*}
  S(t-s)\int_0^sS(s-r)G(u(r))d\omega(r)+\int_s^tS(t-r)G(u(r))d\omega(r)=\int_0^t S(t-r)G(u(r))d\omega(r)
\end{equation*}
where these integrals are defined in the sense of \eqref{sol}.
\end{lemma}

\begin{proof}

Replacing  $(u,v)$ by an approximating sequence $(u^n,(u^n\otimes \omega^n))$ then the above formula holds which follows because
for $F(\cdot)=G(u^n(\cdot))$ the integral can be defined in the sense of Lemma \ref{locall21} and $(u^n\otimes \omega^n)$ can be defined by \eqref{localeq1}. Then the additivity of that integral follows by the method in Z{\"a}hle \cite{Zah98}. Now it suffices to rewrite these integrals according to \eqref{sol} and consider their $V$--limit for $n\to\infty$.
\end{proof}

\begin{remark}Let us stress that for smooth elements $\omega$ both integrals of (\ref{localeq6}) are additive for elements $(u,v)\in C_{\beta,\sim}([0,T];V)\times C_{\beta+\beta^\prime, \sim}(\Delta_{0,T},V\otimes V)$ satisfying (\ref{chen}). But for the sake of conciseness we did not to include the proof in the paper.

We also would like to emphasize that some of the results that we are going to establish below still remain true when considering instead $\hat W_{0,T}$ the bigger space $W_{0,T}$. Nevertheless, for uniqueness of presentation we always keep the space $\hat W_{0,T}$.
\end{remark}

\smallskip

In order to study the existence of solutions to (\ref{localeq01}) we consider the operator
\begin{equation}\label{T}
  \tT(U,\omega,(\omega\otimes_S\omega),u_0)=(\tT_1(U,\omega,u_0),\tT_2(U,\omega,(\omega\otimes_S\omega),u_0))
\end{equation}
defined for $U=(u,v)\in {\hat W_{0,T}}$ by the expressions
\begin{align*}
  \tT_1(U,\omega,u_0)(t)
  &:=
S(t)u_0+(-1)^\alpha\int_0^t\hat D_{0+}^\alpha (S(t-\cdot)G(u(\cdot)))[r]D_{t-}^{1-\alpha}\omega[r]dr\\
  &-(-1)^{2\alpha-1}
  \int_0^tD_{0+}^{2\alpha-1}(S(t-\cdot)DG(u(\cdot)))[r]D_{t-}^{1-\alpha}\dD_{t-}^{1-\alpha}v[r]dr,\\
  \tT_2(U,\omega,(\omega\otimes_S\omega),u_0)(s,t) &:=(-1)^\alpha\int_s^tD_{s+}^\alpha((S(\cdot-s)-{\rm id})u(s))[\xi]\otimes_VD_{t-}^{1-\alpha}\omega[\xi]d\xi\\
    &-(-1)^\alpha\int_s^t \hat D_{s+}^\alpha
    G(u(\cdot))[r] D_{t-}^{1-\alpha}(\omega\otimes_S\omega)(\cdot,t)_{t-}[r]dr\\
&+(-1)^{2\alpha-1}\int_s^t D_{s+}^{2\alpha-1} DG(u(\cdot))[r] D_{t-}^{1-\alpha}\dD_{t-}^{1-\alpha}w(t,\cdot,\cdot)[r]dr,
\end{align*}
where $w$ is given, for $0 < s \leq q\leq  t\leq T$ and $\tilde E\in L_2(V\otimes V,\hat V)$, by
\begin{align}\label{wholder}
\begin{split}
&\tilde Ew(t,s,q)=-\int_s^q\hat D_{s+}^{\alpha}\omega_S(\cdot,t)\tilde E(u(\cdot)-u(s),\cdot)[r]  D_{q-}^{1-\alpha}\omega_{q-}[r]dr\\
&+(-1)^{\alpha-1}\int_s^q D_{s+}^{2\alpha-1}\tilde E(u(\cdot)-u(s),\cdot)[r]  D_{q-}^{1-\alpha}\dD_{q-}^{1-\alpha}(\omega_S(t)\otimes\omega)[r]
    dr\\&+(-1)^{\alpha-1}\int_s^q D_{s+}^{2\alpha-1}\omega_S(\cdot,t)[r] \tilde E D_{q-}^{1-\alpha}\dD_{q-}^{1-\alpha}v(\cdot,q)[r] dr,
\end{split}
\end{align}
being $(\omega_S(t)\otimes\omega)$ defined, for $s\le\tau\le t$ and $E\in L_2(V,\hat V)$, by
\begin{align}\label{oS}
\begin{split}
& E(\omega_S(t)\otimes\omega)(s,\tau)=\omega_S(\tau,t)\int_s^\tau(S(\tau-r)-{\rm id})Ed\omega(r)\\
    +&   (-1)^{-\alpha}E(\omega\otimes_S\omega)(s,\tau)+(\omega_S(\tau,t)-\omega_S(s,t))E(\omega(\tau)-\omega(s)),
\end{split}
\end{align}
where the last integral is defined by fractional derivatives due to the regularity of the semigroup $S$.

Having in mind the Definition \ref{defi}, the corresponding definition of a mild solution of (\ref{localeq01}) is given as follows:
\begin{definition}
$U\in {\hat W_{0,T}}$ such that $U=\tT(U,\omega,(\omega\otimes_S\omega),u_0)$ is called a mild path--area solution to \eqref{localeq00}.
\end{definition}

In what follows we want to study the existence of a fixed point for the operator $\tT$.

\begin{lemma}\label{locall19b}
Assume that $u_0\in V$. There exists a $c>0$ independent of $s<t,\,u_0$ and $\omega$ such that for $U\in {\hat W_{0,T}}$
\begin{align*}
\begin{split}
  \|\tT_1(U,\omega,u_0)\|_{\beta,\sim}\le & c(|u_0|+T^{\beta^\prime}((1+\|u\|_{\beta,\sim}^2)|||\omega|||_{\beta^\prime}+\|u\|_{\beta,\sim}\|v\|_{\beta+\beta^\prime,\sim}))\\\le & c(|u_0|+T^{\beta^\prime}(1+|||\omega|||_{\beta^\prime}){(1+\|U\|_W^2))}.
\end{split}
\end{align*}
\end{lemma}
\begin{proof}
By $\eqref{eq1}$ and $\eqref{eq2}$, for $0< s<t \leq T$ we have that $|(S(t)-S(s))u_0| \leq c \frac{(t-s)^\beta}{s^\beta} |u_0|$,  and then $\|S(\cdot)u_0\|_{\beta,\sim} \le c|u_0|$.
We now use the following abbreviations
\begin{align*}
C_{11}(s,t)&:=(-1)^\alpha\int_s^t \hat D_{s+}^\alpha (
S(t-\cdot)G(u(\cdot)))[r]D_{t-}^{1-\alpha}\omega_{t-}[r]dr, \\
C_{12}(s,t)&:=-(-1)^{2\alpha-1}\int_s^t D_{s+}^{2\alpha-1}(
S(t-\cdot)DG(u(\cdot)))[r] D_{t-}^{1-\alpha}\mathcal{D}
_{t-}^{1-\alpha}v[r]dr, \\
C_{21}(0,s,t)&:=(-1)^\alpha\int_0^s \hat D_{0+}^\alpha
((S(t-\cdot)-S(s-\cdot))G(u(\cdot)))[r]D_{s-}^{1-\alpha}\omega_{s-}[r]dr, \\
C_{22}(0,s,t)&:=-(-1)^{2\alpha-1}\int_0^s
D_{0+}^{2\alpha-1}((S(t-\cdot)-S(s-\cdot))DG(u(\cdot)))[r]
D_{s-}^{1-\alpha}\mathcal{D} _{s-}^{1-\alpha}v[r]dr.
\end{align*}
We are going to estimate the $V$-norm
of the following expression:
\begin{align*}
\int_0^t&S(t-r)G(u(r))d\omega(r)
-\int_0^sS(s-r)G(u(r))d\omega(r)
\\
=& \int_s^t S(t-r)G(u(r)) d\omega(r)+ \int_0^s
(S(t-r)-S(s-r))G(u(r)) d\omega(r) \\
=&C_{11}(s,t)+C_{12}(s,t)+C_{21}(0,s,t)+C_{22}(0,s,t)
\end{align*}
where this equality holds by Lemma \ref{localln1}.
For the first expression, for an $\alpha^\prime$ with $\alpha^\prime-\alpha>0$ sufficiently small we obtain
\begin{align*}
  |C_{11}&(s,t)| \le \frac{1}{\Gamma(1-\alpha)}\int_s^t\bigg(\frac{|S(t-r)G(u(r))|}{(r-s)^\alpha}+\alpha\int_s^r\frac{|(S(t-r)-S(t-q))G(u(r))|}{(r-q)^{1+\alpha}}dq\\
  +&\alpha\int_s^r\frac{|S(t-q)(G(u(r))-G(u(q))-DG(u(q))(u(r)-u(q)))|}{(r-q)^{1+\alpha}}dq\bigg)|||\omega|||_{\beta^\prime}(t-r)^{\alpha+\beta^\prime-1}dr\\
  \le&
  c\int_s^t\bigg(\frac{c_G+c_{DG}|u(r)|}{(r-s)^\alpha}
  +\int_s^r\frac{(r-q)^{\alpha^\prime}(c_G+c_{DG}|u(r)|)}{(t-r)^{\alpha^\prime}(r-q)^{1+\alpha}}dq\\
  +&
  \int_s^r\frac{c_{D^2G}\|u\|_{\beta,\sim}^2(r-q)^{2\beta}}{q^{2\beta}(r-q)^{1+\alpha}}dq\bigg)|||\omega|||_{\beta^\prime}(t-r)^{\alpha+\beta^\prime-1}dr.
\end{align*}
Now evaluating these integrals by \eqref{localeq4} we see that
$$\sup_{0<s<t\leq T}s^\beta \frac{|C_{11}(s,t)|}{(t-s)^\beta}\leq c |||\omega|||_{\beta^\prime} T^{\beta^\prime}(1+\|U\|_W^2).$$

Now let us consider $C_{22}(0,s,t)$. By using Lemma \ref{l0} and (\ref{eq26}) we have that
\begin{align*}
  |& C_{22}(0,s,t)|  \le \frac{c(||v||_{\beta+\beta^\prime,\sim}+\|u\|_{\beta,\sim}|||\omega|||_{\beta^\prime})}{\Gamma(2-2\alpha)}\int_0^s\bigg(\frac{|(S(t-r)-S(s-r))DG(u(r))|}{r^{2\alpha-1}} \\
   & +(2\alpha-1)\int_0^r\frac{|(S(t-r)-S(t-q)-(S(s-r)-S(s-q)))DG(u(r))|}{(r-q)^{2\alpha}}dq\\
   &+(2\alpha-1)\int_0^r\frac{|(S(t-q)-S(s-q))(DG(u(r))-DG(u(q)))|}{(r-q)^{2\alpha}}dq\bigg)\frac{(s-r)^{2\alpha+\beta+\beta^\prime-2}}{r^\beta}dr\\
   \le& c (||v||_{\beta+\beta^\prime,\sim}+\|u\|_{\beta,\sim}|||\omega|||_{\beta^\prime})\int_0^s\bigg(\frac{ c_{DG}(t-s)^\beta}{(s-r)^{\beta}r^{2\alpha-1}}
   +\int_0^r\frac{c_{DG} (r-q)^{2\alpha^\prime-1}(t-s)^\beta }{(s-r)^{2\alpha^\prime-1+\beta}(r-q)^{2\alpha}}dq\\
   &+\int_0^r\frac{c_{D^2G}(t-s)^\beta(r-q)^{\beta}\|u\|_{\beta,\sim}}{q^{\beta}(s-r)^\beta(r-q)^{2\alpha}}dq\bigg)\frac{(s-r)^{2\alpha+\beta+\beta^\prime-2}}{r^\beta}dr
\end{align*}
and by {\bf (H1)} and \eqref{localeq4} we obtain that
$$\sup_{0<s<t\leq T}s^\beta \frac{|C_{22}(0,s,t)|}{(t-s)^\beta}\leq c  T^{\beta^\prime}(1+|||\omega|||_{\beta^\prime})(1+||U||_{W}^2).$$
The remaining terms can be estimated in a similar manner. Setting $s=0$ and considering $C_{11}(0,t)+C_{12}(0,t)$ we obtain the same estimate for the norm of $C([0,T];V)$, hence the proof is complete.

\end{proof}

The following conclusion with respect to the regularity of the above integrals holds true:

\begin{corollary}\label{localcoro1}
Let $\beta<\beta^\prime<H$, then there exists a $c>0$ independent of $s<t$ and $\omega$ such that for $U\in {\hat W_{0,T}}$
\begin{equation*}
 \bigg| (-A)^\beta\int_0^tS(t-r)G(u(r))d\omega\bigg|\le c T^{\beta^\prime-\beta}(1+|||\omega|||_{\beta^\prime})(1+||U||_{W}^2).
\end{equation*}
\end{corollary}

We omit the proof of this result since we only would need to estimate $|(-A)^\beta C_{11}(0,t)+(-A)^\beta C_{12}(0,t)|$, where $C_{11}$ and $C_{12}$ were defined in the proof of Lemma \ref{locall19b}. Hence, it suffices to follow the estimates of that result together with the properties of the semigroup $S$.

\begin{lemma}\label{locall23}
Suppose that the conditions of Lemma \ref{locall19b} hold. Then there exists a $c>0$ depending on $\omega$ such that for $U^1,\,U^2\in {\hat W_{0,T}}$
\begin{align*}
&||{\mathcal T}_1(U^1)-{\mathcal T}_1(U^2)||_{\beta,\sim} \le cT^{\beta^\prime}(1+||U^1||_W^2+||U^2||_W^2)||U^1-U^2||_W+c|u_0^1-u_0^2|.
\end{align*}
\end{lemma}
\begin{proof}
Let us denote $\Delta u=u^1-u^2,\,\Delta U=(u^1-u^2,v^1-v^2)$. Then for $r\in [0,T]$ by Lemma \ref{locall2} we have
\begin{align*}
  &\|G(u^1(r))-G(u^2(r))\|_{L_2(V,\hat V)}\le c_{DG}\|\Delta u\|_{\beta,\sim}\\
\end{align*}
and
\begin{align*}
    \|G&(u^1(r))-G(u^1(q))-DG(u^1(q))(u^1(r)-u^1(q))
    \\-&(G(u^2(r))-G(u^2(q))-DG(u^2(q))(u^2(r)-u^2(q)))\|_{L_2(V,\hat V)}\\
    &\le c_{D^2G}
    (\|u^1\|_{\beta,\sim}+\|u^2\|_{\beta,\sim})\|\Delta u\|_{\beta,\sim}\frac{(r-q)^{2\beta}}{q^{2\beta}}\\
    &+c_{D^3G}\|u^2\|_{\beta,\sim}\frac{(r-q)^\beta}{q^\beta}\sup_{r\in[0,T]}|\Delta u(r)|(2\|u^1\|_{\beta,\sim}+\|u^2\|_{\beta,\sim})\frac{(r-q)^\beta}{q^\beta}\\
    &\le c
    (\|u^1\|_{\beta,\sim}+\|u^2\|_{\beta,\sim})\|\Delta u\|_{\beta,\sim}\frac{(r-q)^{2\beta}}{q^{2\beta}} (1+\|u^2\|_{\beta,\sim}).
    \end{align*}
Now we can follow the proof of Lemma \ref{locall19b}.
\end{proof}

Up to now we have obtained appropriate estimates for the first component $\tT_1$ of the operator $\tT$ given by \eqref{T}. Now we aim at getting the corresponding estimates for $\tT_2$, the second component of $\tT$. To this end, we need the two following Lemmata, dealing with estimates of some of the terms appearing in the expression of $\tT_2$.

\begin{lemma}\label{locall24}
Under the Hypothesis ${\bf (H)}$ the following statements hold:

(i) For the mapping
\begin{align*}
&e\in V\mapsto\omega_S(s,t)e=(-1)^{-\alpha}\int_s^t(S(\xi-s)e)\otimes_Vd\omega(\xi)
\end{align*}
the following properties hold true: for $0\le s\le r\le t\le T,\,e\in V$ and $1/3<\beta^\prime <\beta^{\prime\prime}<H$,
\begin{align*}
\|\omega_S(r,t)e-\omega_S(s,t)e\|& \le c(r-s)^{\beta^\prime}(|||\omega|||_{\beta^\prime}+|||\omega|||_{\beta^{\prime\prime}})|e|,\\
\|\omega_S(s,t)e\|&\le c(t-s)^{\beta^\prime}|||\omega|||_{\beta^{\prime}} |e|,\\
\|\omega_S(s,t)(-A)^{\beta^\prime} e\|&\le c|||\omega|||_{\beta^{\prime\prime}}|e|.
\end{align*}

(ii) The mapping
\begin{equation*}
    E\in L_2(V,\hat V) \mapsto S_\omega(s,t)E=  \int_s^t S(t-r)Ed\omega(r)
\end{equation*}
is in $L_2(L_2(V,\hat V),V)$, with norm bounded by $c|||\omega|||_{\beta^\prime}(t-s)^{\beta^\prime}\|E\|_{L_2(V,\hat V)}$.

\end{lemma}

\begin{proof}
Let $\beta^\prime <\beta^{\prime\prime}<H$ such that for
$\alpha<\alpha^{\prime}<1$ we have
    $\beta^\prime+\alpha^{\prime}<1<\beta^{\prime\prime}+\alpha$.
Then
\begin{align}\label{ex1m}
\begin{split}
    \omega_S(r,t)e-\omega_S(s,t)e=&(-1)^{\alpha}\int_r^t(S(\xi-r)e - S(\xi-s)e)\otimes_Vd\omega(\xi)\\
    &-(-1)^{\alpha}\int_s^rS(\xi-s)e\otimes_Vd\omega(\xi).
    \end{split}
\end{align}
Interpreting these integrals in a fractional sense and using Lemma \ref{l0} we obtain
\begin{align*}
    |D^{\alpha}_{r+}(S(\cdot-r)e-S(\cdot-s)e)[\xi]|&\le c \bigg(\frac{(r-s)^{\beta^\prime}}{(\xi-r)^{\alpha +\beta^\prime}}
    +\alpha
    \int_r^\xi\frac{(r-s)^{\beta^\prime}(\xi-q)^{\alpha^{\prime}}}
    {(\xi-q)^{1+\alpha}(q-r)^{\alpha^{\prime}+\beta^\prime}}dq\bigg)|e|\\
    &\le c(r-s)^{\beta^\prime}(\xi-r)^{-\alpha -\beta^\prime}|e|.
\end{align*}
Moreover, since $\beta^{\prime\prime}< H$, due to (H3) $\omega\in C_{\beta^{\prime\prime}}([0,T];V)$ (and in particular $\omega\in C_{\beta^{\prime}}([0,T],V)$), then by Lemma \ref{locall22}
\begin{align*}
    &|D^{1-\alpha}_{t-}\omega_{t-}[\xi]|\le c|||\omega|||_{\beta^{\prime\prime}}(t-\xi)^{\alpha+\beta^{\prime\prime}-1}.
\end{align*}
Hence, by applying \eqref{localeq4}, the first integral on the
right-hand side of \eqref{ex1m} is bounded in particular by
$c(r-s)^{\beta^\prime} |||\omega|||_{\beta^{\prime\prime}}|e|$.

Furthermore, for the last term in \eqref{ex1m}, thanks to {\bf (H1)} and the regularity properties of the semigroup we obtain
\begin{align*}
    \int_ s^r
    \|D^{\alpha}_{s+}&S(\cdot-s)e[\xi]\otimes_VD^{1-\alpha}_{r-}\omega_{r-}[\xi]\|d\xi
   \le c|||\omega|||_{\beta^{\prime}}(r-s)^{\beta^\prime}|e|.
\end{align*}
The second statement of (i) follows directly from the first one taking $r=t$. For the last conclusion of (i), for the parameters chosen at the beginning of the proof,
  \begin{align*}
   |D_{s+}^{\alpha}(S(\cdot-s)&(-A)^{\beta^\prime}e)[\xi]|\le  c\bigg(\frac{|e|}{
(\xi-s)^{\alpha+\beta^\prime}}\\
&+\int_s^\xi\frac{(\xi-q)^{\alpha^{\prime}}|e|}{(\xi-q)^{1+\alpha}(q-s)^{\alpha^{\prime}+\beta^\prime}}dq\bigg)\leq c |e| (\xi-s)^{-\alpha-\beta^\prime},
\end{align*}
and therefore, since in particular $\beta^\prime+\alpha <1$,
\begin{align*}
    \int_ s^t
    \|D_{s+}^{\alpha}&(S(\cdot-s)(-A)^{\beta^\prime}e)[\xi]\otimes_VD^{1-\alpha}_{t-}\omega_{t-}[\xi]\|d\xi\\
   \leq &c|||\omega|||_{\beta^{\prime \prime }} |e| \int_
    s^t  (\xi-s)^{-\alpha-\beta^\prime}    (t-\xi)^{\alpha+\beta^{\prime \prime}-1}d\xi
    \\\le& c|||\omega|||_{\beta^{\prime \prime}}(t-s)^{\beta^{\prime \prime} -\beta^\prime}|e| \leq  c |||\omega|||_{\beta^{\prime \prime}} |e|.
\end{align*}
Now we prove (ii). First of all, note that $S_\omega$ is well-defined since, thanks to Remark \ref{localremark}, it suffices the H\"older--continuity of $r\mapsto S(t-r)E$ on any interval $[\eps,t]$.
In addition by the embedding $\hat V\subset V$ the mapping
\begin{equation*}
    L_2(V, \hat V) \ni E\mapsto ED_{t-}^{1-\alpha}\omega_{t-}[r]\in V
\end{equation*}
is in $L_2(L_2(V,\hat V),V)$, and the norm can be estimated  $c_{V,\hat V}|D_{t-}^{1-\alpha}\omega_{t-}[r]|$,  see Lemma \ref{locall32} (ii). Furthermore, the integrand $D^\alpha_{s+} S(t-\cdot) [r]\cdot D^{1-\alpha}_{t-} \omega_{t-} [r]$ is weakly measurable  with respect to $L_2(L_2(V,\hat V),V)$ such that by Pettis' theorem
the integrand is measurable. Because of $D_{t-}^{1-\alpha}E \omega_{t-} [r]=ED_{t-}^{1-\alpha} \omega_{t-} [r]$, we get
\begin{align*}
   \|S_\omega(s,t)\cdot\|_{L_2(L_2(V,\hat V),V)}&=\bigg\| \int_s^t S(t-r)\cdot d\omega(r)\bigg\|_{L_2(L_2(V,\hat V),V)}\\& \leq  \int_s^t \| D^\alpha_{s+} S(t-\cdot) [r]\|_{L(V)} \|\cdot D^{1-\alpha}_{t-} \omega_{t-} [r]\|_{L_2(L_2(V,\hat V),V)}dr\\
   & \le c|||\omega|||_{\beta^{\prime}} (t-s)^{\beta^\prime}.
   \end{align*}
\end{proof}

\begin{corollary}\label{coro1}
Let $(\omega^n)_{n\in\NN}$ be a sequence converging to $\omega$ in $C_{\beta^\prime}([0,T];V)$. Then
\[
\lim_{n\to\infty}\sup_{\begin{array}{c}0\le s<t\le T\\
|e|=1\end{array}}\frac{\|(\omega-\omega^n)_S(s,t)e\|}{(t-s)^{\beta^\prime}}=0,\]
\[\lim_{n\to\infty}\sup_{0\le s<t\le T}\frac{\|S_{\omega-\omega^n}(s,t)\cdot\|_{L_2(L_2(V,\hat V),V)}}{(t-s)^{\beta^\prime}}=0.
\]
From the previous result we also obtain that
\[
D^{\alpha}_{s+}(S(\cdot-s)(S_\omega(q,s)\cdot))[\xi]\otimes_VD^{1-\alpha}_{r-}\omega_{r-}[\xi]
\]
is weakly $L_2(L_2(V,\hat V),V\otimes V)$--measurable.
\end{corollary}

\begin{lemma}\label{locall9} Suppose that the hypothesis {\bf (H)} holds. For $0 < s\leq q\le  t\le T$, $\tilde E\in L_2(V\otimes V,\hat V)$  and $U=(u,v)\in
{\hat W_{0,T}}$ consider the mapping
\begin{equation*}
 w(t,s,q):   L_2(V\otimes V,\hat V)\ni \tilde E\mapsto \tilde Ew(t,s,q)\in V\otimes V
\end{equation*}
given by (\ref{wholder}) and \eqref{oS}. Then it is well-defined and satisfies, for $\beta^\prime< \beta^{\prime\prime}$ the estimate
\begin{align*}
\begin{split}
&\|w(t,s,q)\|_{L_2(L_2(V\otimes V,\hat V),V\otimes V)}\le  c ||U||_W s^{-\beta}(q-s)^{\beta+\beta^\prime}(t-s)^{\beta^\prime},
    \end{split}
    \end{align*}
where the constant $c$ depends on $|||\omega|||_{\beta^{\prime\prime}}$ and $\|(\omega\otimes_S\omega)\|_{2\beta^\prime}$.
In particular
\begin{equation*}
  c\sim  (1+|||\omega|||_{\beta^{\prime\prime}}^2)\|(\omega\otimes_S\omega)\|_{2\beta^\prime}.
\end{equation*}
\end{lemma}
For the proof, see the Appendix.\\

We would like to point out that $w$ satisfies the generalized Chen--equality
\begin{align}\label{localeq9b}
\begin{split}
\tilde E &w(t,s,r)+\tilde E w(t,r,q)-\tilde E(u(r)-u(s),\cdot)(\omega\otimes_S\omega)(r,q) \\
    &=\tilde E w(t,s,q)+\omega_S(q,t)S_\omega(q,r)\tilde E(u(r)-u(s),\cdot),
    \end{split}
\end{align}
for $0 < s\le r \leq q\le  t\le T$, and the Chen--equality
\begin{align}\label{localeq10b}
\begin{split}
\tilde Ew(t,s,r)+\tilde Ew(t,r,t)
    -\tilde E(u(r)-u(s),\cdot)(\omega\otimes_S\omega)(r,t)
    =\tilde Ew(t,s,t)
    \end{split}
\end{align}
where the latter one is obtained from (\ref{localeq9b}) taking $q=t$. In order to prove these two properties we only need to follow an approximation argument, considering $(\omega^n,(\omega^n\otimes_S\omega^n))$ satisfying {\bf (H3)}, and therefore converging to $(\omega,(\omega\otimes_S\omega))$ in $C_{\beta^\prime}([0,T], V)\times C_{2\beta^\prime}(\bar \Delta_{0,T};L_2(L_2(V,\hat V), V\otimes V))$ and take into account that the approximating elements $w^n:=(u\otimes (\omega^n\otimes_S \omega^n))$ given in Lemma \ref{uomom} satisfy the Chen--equalities (\ref{localeq9}) and (\ref{localeq10}). In particular, the terms $(\omega^n_S \otimes \omega^n)$ converge to the corresponding term $(\omega_S \otimes \omega)$, see the proof of Lemma \ref{locall9} in the Appendix section.\\

The following result gives estimates for the fractional derivatives of $v$ and $w$.

\begin{lemma}\label{locall3}
Let $U=(u,v)\in {\hat W_{0,T}}$. Then for $0<r<t\le T$ we have
\begin{align*}
  &\|D_{t-}^{1-\alpha}\dD_{t-}^{1-\alpha}v[r]\|_{\beta+\beta^\prime,\sim}\le c ||U||_W r^{-\beta}(t-r)^{2\alpha+\beta+\beta^\prime-2},\\
  & \|D_{t-}^{1-\alpha}\dD_{t-}^{1-\alpha}w(t,\cdot,\cdot)[r]\|\le c ||U||_Wr^{-\beta}(t-r)^{2\alpha+\beta+2\beta^\prime-2},
\end{align*}
where the first constant $c$ depends on $|||\omega|||_{\beta^{\prime\prime}}$, and the second one on $|||\omega|||_{\beta^{\prime\prime}}$ and $\|(\omega\otimes_S\omega)\|_{2\beta^\prime}$.
\end{lemma}
We have also shifted the proof of this result to the Appendix section.
\medskip

\begin{lemma}\label{locall20}
Assume that $u_0 \in V$. There exists positive constants $\tilde c$ and $c$ such that for $U\in {\hat W_{0,T}}$
\begin{align*}
  \|\tT_2(U,\omega, & (\omega\otimes_S\omega),u_0)\|_{\beta+\beta^\prime,\sim}\le  \tilde c |u_0|+
  cT^{\beta^\prime}(1+\|U\|_W^2)
  \end{align*}
and, in addition, for two elements $U^1,\,U^2\in {\hat W_{0,T}}$:
\begin{align*}
  \|\tT_2 &(U^1,\omega,(\omega\otimes_S\omega),u_0)-\tT_2(U^2,\omega,(\omega\otimes_S\omega),u_0)\|_{\beta+\beta^\prime,\sim}\\
  \le & \tilde c|u_0^1-u_0^2| +cT^{\beta^\prime}(1+\|U^1\|_W^2+\|U^2\|_W^2)\|U^1-U^2\|_W.
\end{align*}
The constant $c$ depends on $|||\omega|||_{\beta^{\prime\prime}}$, $|||\omega|||_{\beta^{\prime}}$ and $\|(\omega\otimes_S\omega)\|_{2\beta^\prime}$, while $\tilde c$ on $|||\omega|||_{\beta^{\prime}}$.
\end{lemma}
\begin{proof}
Let us denote ${\mathcal T}_2(U)(s,t)=: B_1(s,t)+B_2(s,t)+B_3(s,t)$,
corresponding to the three different addends of $\tT_2.$

For $B_1$ we can consider the following splitting:
\begin{align*}
B_1(s,t)=&\int_s^t (S(\xi-s)-{\rm id})u(s)\otimes_V d\omega(\xi)\\
=&\int_s^t (S(\xi)-S(s))u_0 \otimes_V d\omega(\xi)\\
&+\int_s^t
(S(\xi-s)-{\rm id})\bigg(\int_0^s S(s-r)G(u(r))d\omega (r)
\bigg)\otimes_V d\omega(\xi)
\\=:&B_{11}(s,t)+B_{12}(s,t).
\end{align*}
$B_1$ can be interpreted in the fractional  sense thanks to the regularity of its integrand, which means that
\begin{align*}
B_{11}(s,t)& =(-1)^\alpha \int_s^t
D^\alpha_{s+}((S(\cdot)-S(s))u_0)[\xi]\otimes_VD^{1-\alpha}_{t-}\omega_{t-}[\xi]d\xi.
\end{align*}
For $\alpha<\alpha^\prime$ where $\alpha^\prime$ is sufficiently close to $\alpha$ and $s>0$, applying \eqref{eq1} and
\eqref{eq2},
\begin{align*}
    |D^\alpha_{s+}((S(\cdot)&-S(s))u_0)[\xi]|\le c\bigg(\frac{|(S(\xi)-S(s))u_0|}{(\xi-s)^{\alpha}}+\int_s^\xi
    \frac{|(S(\xi)-S(q))u_0|}{(\xi-q)^{1+\alpha}}dq\bigg)\\
    &\le c\bigg(\frac{(\xi-s)^{\beta-\alpha}}{s^\beta}
    +\int_s^\xi\frac{(\xi-q)^{\alpha^\prime}q^{-\beta}}{(\xi-q)^{1+\alpha}q^{\alpha^\prime-\beta}}dq\bigg)|u_0|\\
    &\le c\bigg(\frac{(\xi-s)^{\beta-\alpha}}{s^\beta}
    +\frac{1}{s^\beta}\int_s^\xi\frac{(\xi-q)^{\alpha^\prime}}{(\xi-q)^{1+\alpha}(q-s)^{\alpha^\prime-\beta}}dq\bigg)|u_0|\le c\frac{(\xi-s)^{\beta-\alpha}}{s^\beta}|u_0|,
\end{align*}
hence, for $s>0$,
\begin{align*}
|B_{11}(s,t)|&  \leq c \frac{|||\omega|||_{\beta^\prime}
|u_0|}{s^\beta} \int_s^t
(\xi-s)^{\beta-\alpha}(t-\xi)^{\beta^\prime+\alpha-1} d\xi\leq  c |||\omega|||_{\beta^\prime}
|u_0| s^{-\beta} (t-s)^{\beta^\prime+\beta},
\end{align*}
which implies $\|B_{11}\|_{\beta+\beta^\prime,\sim}  \leq \tilde c |u_0|$. Moreover, note that
\begin{align*}
\begin{split}
\bigg|D_{s+}^\alpha &\bigg((S(\cdot-s)-{\rm id})\int_0^s
S(s-r)G(u(r))d\omega
(r)\bigg)[\xi]\bigg| \\
\leq &c \frac{|S(\xi-s)-{\rm id}) \int_0^s S(s-r)G(u(r))d\omega
(r)| }{(\xi-s)^\alpha}\\ & + \int_s^\xi \frac{|\int_0^s
S(\xi-r)-S(q-r))G(u(r))d\omega(r)|}{(\xi-q)^{1+\alpha}}dq\bigg).
\end{split}
\end{align*}
To deal with $B_{12}$ on account of Corollary \ref{localcoro1} and thanks to the fact that $\beta^\prime>\beta$ we have
\begin{align*}
\int_s^\xi& \frac{|\int_0^s (S(\xi-q)-{\rm id})S(q-r)G(u(r))d\omega(r)|}{(\xi-q)^{1+\alpha}}dq \\
&\leq  \int_s^\xi \frac{(\xi-q)^{\alpha^\prime}}{(q-s)^{\alpha^\prime-\beta}(\xi-q)^{1+\alpha}}\bigg|(-A)^\beta\int_0^s S(s-r)G(u(r))d\omega(r)\bigg|dq\\
&\leq c (1+|||\omega|||_{\beta^\prime})(1+\|U\|_W^2)
(\xi-s)^{\beta-\alpha}s^{\beta^\prime-\beta}
\end{align*}
and by \eqref{localeq4}
\begin{align*}
\|B_{12}\|_{\beta+\beta^\prime,\sim}\leq
cT^{\beta^\prime}|||\omega|||_{\beta^\prime}^2  (1+\|U\|_W^2).
\end{align*}
Finally, a similar estimate follows for $B_2$ and $B_3$. In order to see this, note that $B_2$ and $B_3$ can be considered in a similar way to $C_{11}$ and $C_{12}$ in the proof of Lemma \ref{locall19b}, with the difference that now we have to estimate $D_{t-}^{1-\alpha}(\omega\otimes_S\omega)(\cdot, t)_{t-}[r]$ and
$D_{t-}^{1-\alpha}
\dD_{t-}^{1-\alpha}w(t,\cdot,\cdot)[r]$, for which we use the $2\beta^\prime$-H\"older continuity of $(\omega\otimes_S\omega)$ together with Lemma \ref{locall3}, arriving at
\begin{align*}
&\|B_{2}\|_{\beta+\beta^\prime,\sim}\leq
cT^{\beta^\prime}\|(\omega\otimes_S\omega)\|_{2\beta^\prime} (1+\|U\|_W^2),\\
&\|B_{3}\|_{\beta+\beta^\prime,\sim}\leq
cT^{\beta^\prime} (1+|||\omega|||_{\beta^{\prime\prime}}^2)\|(\omega\otimes_S\omega)\|_{2\beta^\prime}(1+\|U\|_W^2).
\end{align*}

The second part of this lemma can be proven similarly and thus we omit it here.
\end{proof}

\medskip

Now we approximate $\tT(U,\omega,(\omega\otimes_S\omega),u_0)$ by piecewise linear noise.

\begin{lemma}\label{locall5}
Let $U\in {\hat W_{0,T}}$ and assume that $(\omega^n,(\omega^n\otimes_S\omega^n))$ satisfies {\bf(H3)}. Then
\begin{equation*}
  \lim_{n\to\infty}\|\tT(U,\omega^n,(\omega^n\otimes_S\omega^n),u_0)-\tT(U,\omega,(\omega\otimes_S\omega),u_0)\|_W=0.
\end{equation*}
In addition, $\tT(U,\omega,(\omega\otimes_S\omega),u_0)\in \hat W_{0,T}$.
\end{lemma}
\begin{proof}
Note that if a term of $\tT$ contains $\omega$ or $(\omega\otimes_S\omega)$ then this term depends on these expressions linearly or bi-linearly (see the definition of $\tT_1,\,\tT_2$ together with \eqref{wholder}) which just gives the convergence conclusion.
To see the second part of the statement we take $U=(u,v)\in  \hat W_{0,T}$. For this element we choose an approximating sequence $(u^n,(u^n\otimes\omega^n))$ from \eqref{localeq12}. We note that $\tT_1((u^n,(u^n\otimes\omega^n)),\omega^n,u_0^n)\in C_{\gamma}([0,T],V)$
and $\tT_1((u^n,(u^n\otimes\omega^n)),\omega^n,u_0^n)(0)\in D((-A))$ for any $\gamma\in (0,1)$, see Pazy \cite{Pazy} Theorem 4.3.1 and \eqref{eq1}-(\ref{eq2}). Therefore $\tT_2((u^n,(u^n\otimes\omega^n)),\omega^n,(\omega^n\otimes_S \omega^n),u_0^n)$ can be defined as $(\tT_1((u^n,(u^n,\omega^n)),\omega^n, u_0^n)\otimes \omega^n)     $ given by (\ref{localeq1}). By definition of the space $\hat W_{0,T}$, by Lemmas \ref{locall23}, \ref{locall20}, and the first part of this lemma we have that that $\tT(U,\omega,(\omega\otimes_S\omega),u_0)\in \hat W_{0,T}$.
\end{proof}

Let us now prove the uniqueness of the path-area-solution in$\hat W_{0,T}$ if such a solution exists.

\begin{theorem}
Suppose that $U^1=(u^1,v^1),\,U^2=(u^2,v^2)\in \hat W_{0,T}$ are two path-area solutions related to the initial condition $u_0\in V$. Then we have $U^1=U^2$.
\end{theorem}
\begin{proof}
Let $T_1\in [0,T)$ such that $[0,T_1]$ represents the maximal interval of uniqueness. Here $T_1=0$ means that two different solutions are just branching from $0$.
Note that the restrictions $\tilde U^i=(u_{[T_1,\tilde T]}^i,v_{\Delta_{[T_1,\tilde T]}}^i)$ are path-area solutions too with respect to the interval $[T_1,\tilde T],\,\tilde T \leq T$ with initial condition $u(T_1)$.
This follows because for $\tT_1$ we can apply Lemma \ref{localln1}. On the other hand, $\tT_2$ restricted to $T_1<s<t\le \tilde T$ keeps the same structure as the original $\tT_2$. Then, if the constant $C$ is an estimate of $\|\tilde U^i\|_{W_{T_1,\tilde T}}^2$, similar to the estimates that we obtained in Lemmas \ref{locall23} and \ref{locall20}, we obtain
\begin{equation*}
  0\not=\|\tilde U^1-\tilde U^2\|_{W_{T_1,\tilde T}}\le c(\tilde T-T_1)^{\beta^\prime}(1+2C)\|\tilde U^1-\tilde U^2\|_{{W_{T_1,\tilde T}}}
\end{equation*}
or equivalently
\begin{equation*}
  1\le c(\tilde T-T_1)^{\beta^\prime}(1+2C)
\end{equation*}
for any $\tilde T>T_1$, which is a contradiction if $\tilde T-T_1$ is sufficiently small.
\end{proof}

Now we present the main theorem of the paper.

\begin{theorem}\label{localt1}
Suppose that the standing conditions {\bf (H)} are satisfied and suppose that $T>0$ is chosen sufficiently small depending on the parameters of the problem. Then $\tT$ has a fixed point in $\hat W_{0,T}\subset W_{0,T}$ that defines a mild path-area solution to \eqref{localeq01} which is unique.
\end{theorem}
\begin{proof}
Lemmas \ref{locall19b}, \ref{locall20} and \ref{locall5} prove that $\tT$ maps a closed centered ball from $\hat W_{0,T}$ into itself if $T>0$ is sufficiently small and
by Lemma \ref{locall23} and Lemma \ref{locall20} we obtain that this mapping is a contraction
if $T>0$ is chosen sufficiently small.
\end{proof}

\begin{corollary}
Let $U\in \hat W_{0,T}$ be a mild path-area solution given by Theorem \ref{localt1} and let $U^n$ be the path-area solution corresponding to the equations (\ref{localeq01}) and \eqref{eq3}
and driven by a smooth noise $\omega^n$. Then we have
\begin{equation*}
  \lim_{n\to\infty}\|U^n-U\|_W=0.
\end{equation*}
\end{corollary}
In fact, for large enough $n$ we can find a ball in $\hat W_{0,T}$ which is mapped into itself by $\tT(\cdot,\omega^n,(\omega^n\otimes_S\omega^n),u_0)$ and
$\tT(\cdot,\omega,(\omega\otimes_S\omega),u_0)$, and where in addition these mappings are contractions with uniform contraction condition.
Then the conclusion follows by the parameter version of the Banach-fixed point Theorem together with Lemma \ref{locall5}.

\begin{remark}
An application of the main results of this article is to prove that the (pathwise) mild path--area solutions are global, see the forthcoming paper \cite{GLS12}. In addition, the results presented in this paper are the basement to prove that SEEs with non-trivial diffusion coefficients $G$ and driven by fractional Brownian--motions $B^H$ with $H\in (1/3,1/2]$ generate a random dynamical system, see also \cite{GLS12}. That property has been established recently by the same authors when dealing with a more regular fBm, namely $B^H$ with $H\in (1/2,1)$, see \cite{GLS09}.
\end{remark}

\section{Example}
Let $V=l_2$  be the space
of square additive sequences with values in $\RR$. In addition, let
$A$ be a negative symmetric operator defined on $D(-A)\subset l_2$
with compact inverse. In particular, we can assume that $-A$ has a
discrete spectrum $0<\lambda_1\le \lambda_2\le\cdots\le
\lambda_i\le \cdots\to\infty$  where the associated
eigenelements $(e_i)_{i\in\NN}$ form a complete orthonormal system
in $l_2$. The
spaces $D((-A)^\nu)=V_\nu$ are then defined by
\begin{equation*}
    \{u=(u_i)_{i\in\NN}\in l_2: \sum_i \lambda_i^{2\nu}
    u_i^2=:|u|_{V_\nu}^2<\infty\}.
\end{equation*}

Assume that there exists $\kappa>0$ such that the Hilbert-Schmidt embedding $V_\kappa \subset V$ holds true, and take $\hat V=V_\kappa$. Consider a sequence of functions
$(g_{ij})_{i,j\in \NN}$, with $g_{ij}:V\to \RR$, and define $G(u)$ for $u\in V$ by

\begin{equation*}
    G(u)v=\bigg(\sum_j
    g_{ij}(u)v_j\bigg)_{i\in\NN}\quad\text{for all }v\in V.
\end{equation*}
We assume that
\begin{align*}
\|G(u)\|_{L_2(V,V_\kappa)}^2&=\sum_j |G(u) e_j|_{V_\kappa}^2=\sum_j \bigg(\sum_i \lambda_i^{2\kappa} (G(u) e_j)_i^2\bigg)=\sum_{i,j} \lambda_i^{2\kappa} g_{ij}^2(u)\le c
\end{align*}
uniformly with respect to $u\in V$.

In addition, assume that $g_{ij}$ are four times  differentiable and their  derivatives  are uniformly bounded in the following way
\begin{align*}
&|Dg_{ij}(u)(e_k)|\le c_{g,1}^{ijk},\quad  |D^2g_{ij}(u)(e_k,h_1)|\le c_{g,2}^{ijk}|h_1|,\quad  |D^3g_{ij}(u)(e_k,h_1,h_2)|\le c_{g,3}^{ijk}|h_1||h_2|,\\
&|D^4g_{ij}(u)(e_k,h_1,h_2,h_3)|\le  c_{g,4}^{ijk}|h_1||h_2||h_3|\quad\text{for any }u\in V,
\end{align*}
and these bounds  satisfy
\begin{align*}
&\sum_{ijk}\lambda_i^{2\kappa}(c_{g,1}^{ijk})^2<\infty,\quad  \sum_{ijk}\lambda_i^{2\kappa}(c_{g,2}^{ijk})^2<\infty,\quad  \sum_{ijk}\lambda_i^{2\kappa}(c_{g,3}^{ijk})^2<\infty,\sum_{ijk}\lambda_i^{2\kappa}(c_{g,4}^{ijk})^2<\infty.
\end{align*}
To see for instance that $DG$ exists, note that by Taylor expansion
\begin{align*}
|g_{ij}(u+h)-g_{ij}(u)-Dg_{ij}(u)(h)|^2\le \frac12|  D^2g_{ij}(u+\eta h)(h,h)|^2
    \le (c_{g,2}^{ijk})^2|h|^4
\end{align*}
where $u,\, h\in V$ and $\eta\in [0,1]$.
In particular, we also note that
\begin{equation*}
    \sum_{jk}\bigg|DG(u)(e_k,e_j)\bigg|_{V_\kappa}^2
=\sum_{ijk}\lambda_i^{2\kappa}|Dg_{i,j}(u)(e_k)|^2 \le  \sum_{ijk}\lambda_i^{2\kappa}(c_{g,1}^{ijk})^2=:c_{DG}^2<\infty.
\end{equation*}
This condition ensures the Lipschitz
continuity of $G$ as well as the Hilbert-Schmidt property of   $DG$. Similarly, we obtain that $DG$ is also Lipschitz with respect  to the Hilbert-Schmidt norm. We also obtain the existence of  the second and third derivative. Hence the conditions on $G$ in Hypothesis {\bf (H)} hold.\\

The developed theory can be also applied to other examples of G, like kernel integrals, see \cite{GLS12}.

\section{Appendix}
In the appendix we prove some technical estimates related to $w=(u\otimes(\omega\otimes_S\omega))$.
\medskip

{\bf Proof of Lemma \ref{uomom}}

\begin{proof}
For $\tilde E \in L_2(V\otimes V,\hat V)$ we define $f_{\tilde E}:L_2(V,V\otimes V)\times V\to L_2(V,V \otimes
V)$ given by
\[
f_{\tilde E}(Q,u)=Q(\tilde E (u,\cdot) ).
\]
From \eqref{omegaSS} for smooth $\omega$ we have that
\begin{align*}
\tilde E(u\otimes(\omega\otimes_S\omega)(t))(s,q)=&
-(-1)^{\alpha} \int_s^q
\omega_S(r,t) \tilde
E(u(r)-u(s),\omega^\prime(r)) dr \\
=&-(-1)^{\alpha}\int_s^q f_{\tilde E}(\omega_S(r,t),u(r)-u(s))
\omega^\prime(r)dr.
\end{align*}
Following Theorem 3.3 in \cite{HuNu09}, we have
\begin{align}\label{tensor3}
\begin{split}
\int_s^q &f_{\tilde E}(\omega_S(r,t),u(r)-u(s))
\omega^\prime(r)dr\\=&(-1)^\alpha\int_s^q \hat D_{s+}^\alpha
f_{\tilde
E}(\omega_S(\cdot,t),u(\cdot)-u(s))[r]D^{1-\alpha}_{q-}\omega_{q-}[r]dr\\
&-(-1)^\alpha \int_s^q (\omega_S(r,t)-\omega_S(s,t)) D^\alpha_{s+}
\tilde E(u(\cdot)-u(s),\cdot)[r] D^{1-\alpha}_{q-}\omega_{q-}[r] dr\\
&-(-1)^\alpha \int_s^q D^\alpha_{s+} \omega_S
(\cdot,t) [r]\tilde E(u(r)-u(s),\cdot)D^{1-\alpha}_{q-}\omega_{q-}[r] dr\\
& +\int_s^q Df_{\tilde
E}(\omega_S(r,t),u(r)-u(s))(\omega_S(r,t)-\omega_S(s,t),u(r)-u(s))
\omega^\prime(r)dr.
\end{split}
\end{align}
Now we calculate the derivative of $f_{\tilde E}$:
\begin{align*}
\begin{split}
Df_{\tilde
E}&(\omega_S(r,t),u(r)-u(s))(\omega_S(r,t)-\omega_S(s,t),u(r)-u(s))\omega^\prime(r)\\
    =&(\omega_S(r,t)-\omega_S(s,t))\tilde
    E(u(r)-u(s),\omega^\prime(r))+\omega_S(r,t)\tilde
    E(u(r)-u(s),\omega^\prime(r))\\
    =&\tilde
    E(u(r)-u(s),\cdot)D_2(\omega_S(t)\otimes\omega)(s,r)+\omega_S(r,t)\tilde E D_2(u\otimes\omega)(s,r).
\end{split}
\end{align*}
Substituting the above expression in \eqref{tensor3}, after applying fractional integration (\ref{byparts}) to the last two terms, we have to calculate $ED^{1-\alpha}_{q-}(\omega_S(t)\otimes\omega)_{q-}(s,\cdot)[r]$, for $E\in L_2(V,\hat V)$, and $\tilde E D^{1-\alpha}_{q-} (u\otimes\omega)(t)(s,\cdot)_{q-}[r]$. First, we have

\begin{align*}
ED^{1-\alpha}_{q-}&(\omega_S(t)\otimes\omega)_{q-}(s,\cdot)[r]
=D^{1-\alpha}_{q-}E(\omega_S(t)\otimes\omega)_{q-}(s,\cdot)[r]\\
=&\frac{(-1)^{1-\alpha}}{\Gamma(\alpha)}
\bigg(\frac{E(\omega_S(t)\otimes \omega)(s,r)-E(\omega_S(t)\otimes\omega)(s,q)}{(q-r)^{1-\alpha}}\\
&\qquad + (1-\alpha)\int_r^q
\frac{E(\omega_S(t)\otimes\omega)(s,r)-E(\omega_S(t)\otimes\omega)(s,\theta)}{(\theta-r)^{2-\alpha}}
d\theta\bigg)\\
&=-\frac{(-1)^{1-\alpha}}{\Gamma(\alpha)}
\bigg(\frac{E(\omega_S(t)\otimes\omega)(r,q)+(\omega_S(r,t)-\omega_S(s,t))E(\omega(q)-\omega(r))}{(q-r)^{1-\alpha}}\\
&\qquad + (1-\alpha)\int_r^q
\frac{E(\omega_S(t)\otimes\omega)(r,\theta)+(\omega_S(r,t)-\omega_S(s,t))E(\omega(\theta)-\omega(r))}{(\theta-r)^{2-\alpha}}
d\theta\bigg)\\
&= -E{\mathcal D}^{1-\alpha}_{q-} (\omega_S(t)\otimes
\omega))[r]
+(\omega_S(r,t)-\omega_S(s,t))ED^{1-\alpha}_{q-}\omega_{q-}(r).
\end{align*}
Secondly, similar to (\ref{localas}) we obtain that
\begin{align*}
\tilde E D^{1-\alpha}_{q-} (u\otimes\omega)(t)(s,\cdot)_{q-}[r]=-\tilde
E{\mathcal
D}^{1-\alpha}_{q-} (u\otimes\omega)(t)[r]+\tilde E(u(r)-u(s),D^{1-\alpha}_{q-}\omega_{q-}(r)),
\end{align*}
and substituting the last two expressions into \eqref{tensor3} we obtain the conclusion.
\end{proof}

In the following we want to prove that $w$ is well--defined for $U\in {\hat W_{0,T}}$. To this end we will consider
a mapping
\begin{align*}
w(t,s,q)=w(U,\omega,(\omega\otimes_S\omega))(t,s,q)
\end{align*}
which coincides for smooth $\omega$ with the expression introduced in Lemma \ref{uomom}.

In order to prove the regularity of $w$ stated in Lemma \ref{locall9} we need several properties that we collect and prove in the following result.
\begin{lemma}\label{locall32}
(i) Let $\tilde E\in L_2(V\otimes V,\hat V)$ and let $v\in V\otimes V$ be fixed. Then the mapping
\begin{equation*}
  \tilde E\mapsto \tilde Ev
\end{equation*}
is in $L_2(L_2(V\otimes V,\hat V),V)$.\\
(ii)
Let $E\in L_2(V,\hat V)$ and let $u\in  V$ be fixed. Then the mapping
\begin{equation*}
  E\mapsto  Eu
\end{equation*}
is in $L_2(L_2(V,\hat V),V)$.\\

(iii) Let $\tilde E\in L_2(V\otimes V,\hat V)$ and $u\in V$. Then
\begin{equation*}
  \|\tilde E(u,\cdot)\|_{L_2(V,\hat V)}=\|\tilde E(u\otimes_V\cdot)\|_{L_2(V,\hat V)}\leq |u|\|\tilde E\|_{L_2(V\otimes V,\hat V)}.
\end{equation*}
\end{lemma}
\begin{proof}
Consider the separable Hilbert--space $L_2(V\otimes V,\hat V)$ equipped with the complete orthonormal basis
$(\tilde E_{ijk})_{i,j,k\in \NN}$ given by
\begin{equation*}
     \tilde E_{ijk}(e_l\otimes_Ve_{m})=\tilde E_{ijk}(e_l,e_{m})=\left\{\begin{array}{lcl}
    0&:& j\not= l\;\text{ or } k\not= m\\
    f_i &:& j= l\;\text{ and }  k=m.
    \end{array}
    \right.
\end{equation*}
We remind that $(e_i)_{i \in \NN}$ and $(f_i)_{i\in \NN}$ are, respectively, complete orthonormal basis of $V$ and $\hat V$.
Then for $v\in V \otimes V$,
\begin{equation*}
  \sqrt{\sum_{ijk}(\tilde E_{ijk}v)^2}=\sqrt{\sum_i|f_i|^2\sum_{jk}v_{jk}^2}=c_{V,\hat V}\|v\|
\end{equation*}
where $v_{jk}$ is the mode of $v$ with respect to $e_j\otimes_V e_k$. This completes (i).

The second statement can be proven similarly and therefore we omit its proof.

Finally, we have
\begin{align*}
\begin{split}
  \|\tilde E(u,\cdot)&\|_{L_2(V,\hat V)}^2= \sum_i|\tilde E(u,e_i)|_{\hat V}^2\\
  &=\sum_i|\sum_k u_k \tilde E(e_k,e_i)|_{\hat V}^2
  \le \sum_i\bigg(\sum_k |u_k| |\tilde E(e_k,e_i)|_{\hat V}\bigg)^2  \\
   & \le \sum_i\bigg(\big(\sum_k |u_k|^2\big)^\frac12 \big(\sum_k|\tilde E(e_k,e_i)|_{\hat V}^2\big)^\frac12\bigg)^2=|u|^2\|\tilde E\|_{L_2(V\otimes V,\hat V)}^2.
 \end{split}
\end{align*}
\end{proof}

\medskip

In what follows we abbreviate the notation in the following way: let us denote
\begin{equation*}
L_{2,\otimes}=L_2(L_2(V,\hat V),V\otimes V),\qquad L_{2,\otimes,\otimes}=L_2(L_2(V\otimes V,\hat V),V\otimes V).
\end{equation*}

{\bf Proof of Lemma \ref{locall9}.}

\smallskip

\begin{proof} Let us consider separately the three terms of  $w(t,s,q)$ given by (\ref{wholder}). Precisely, we start estimating the third term
\begin{align*}
I_3(\tilde E):=\int_s^q D_{s+}^{2\alpha-1}\omega_S(\cdot,t)[r]
D_{q-}^{1-\alpha}\dD_{q-}^{1-\alpha}\tilde E v[r]
dr.
\end{align*}
As we have seen in Lemma \ref{locall32} (i), for a fixed $v\in V\otimes V$ the mapping $L_2(V\otimes V,\hat V)\ni \tilde E\mapsto  \tilde Ev$
is in $L_2(L_2(V\otimes V,\hat V),V)$ where an estimate of the norm of this operator is given by $c_{V,\hat V}\|v\|$.
Then, since Lemma \ref{locall24} (i) in particular implies that $D_{s+}^{2\alpha-1}\omega_S(r,t)$ is in $L(V,V\otimes V)$, the mapping $\tilde E\mapsto I_3(\tilde E)$ is in $L_{2,\otimes,\otimes}$.
We have
\begin{align*}
  \|I_3(\cdot)\|_{L_{2,\otimes,\otimes}}& \leq \int_s^q \| D_{s+}^{2\alpha-1}\omega_S(\cdot,t)[r] \cdot D_{q-}^{1-\alpha}{\mathcal D}_{q-}^{1-\alpha} v[r]\|_{L_{2,\otimes,\otimes}}dr\\
  &\leq \int_s^q \|D_{s+}^{2\alpha-1}\omega_S(\cdot,t)[r]\|_{L(V,V\otimes V)}
  \|\cdot D_{q-}^{1-\alpha}{\mathcal D}_{q-}^{1-\alpha} v[r]\|_{L_2(L_2(V\otimes V,\hat V),V)}dr.
  \end{align*}
In order to estimate the second factor in the integrand of $I_3$, note that
due to Lemma \ref{locall3}, for $r\in (s,q)$ and $U\in
   \hat W_{0,T}$,  we obtain
\begin{align*}
    &\|D_{q-}^{1-\alpha}\dD_{q-}^{1-\alpha}v[r]\|\le
    c\|U\|_Wr^{-\beta}(q-r)^{\beta+\beta^\prime +2\alpha-2}
\end{align*}
and as we have said at the beginning of this proof
\begin{align*}
\|D_{q-}^{1-\alpha}{\mathcal D}_{q-}^{1-\alpha}\cdot v[r]\|_{L_2(L_2(V\otimes V,\hat V),V)} \leq cc_{V,\hat V} \|U\|_W r^{-\beta}(q-r)^{\beta+\beta^\prime +2\alpha-2}.
\end{align*}
On the other hand, by Lemma \ref{locall24} (i),
\begin{equation*}
   \|D_{s+}^{2\alpha-1}\omega_S(\cdot,t)[r]\|_{L(V,V\otimes V)}\le c\bigg(\frac{(t-r)^{\beta^\prime}}{(r-s)^{2\alpha-1}}
    +\int_s^r\frac{(r-\xi)^{\beta^\prime}}{(r-\xi)^{2\alpha}}d\xi\bigg)(|||\omega|||_{\beta^\prime}+|||\omega|||_{\beta^{\prime\prime}}).
\end{equation*}
Combining the previous estimates we can conclude
\begin{align*}
 \|I_3(\cdot)\|_{L_{2,\otimes,\otimes}}\le &c\|U\|_Ws^{-\beta}(t-s)^{\beta^\prime}(q-s)^{\beta+\beta^\prime},
\end{align*}
where $c$ depends on $|||\omega|||_{\beta^{\prime\prime}}$  and $|||\omega|||_{\beta^{\prime}}$. Next we deal with
\begin{align*}
I_2(\tilde E):=\int_s^q \hat
D_{s+}^{\alpha}\omega_S(\cdot,t)\tilde
E(u(\cdot)-u(s),\cdot)[r]D_{q-}^{1-\alpha}\omega_{q-}[r]dr.
\end{align*}
Observe that
\begin{align*}
\begin{split}
\|\hat
D_{s+}^{\alpha}&\omega_S(\cdot,t)\tilde E(u(\cdot)-u(s),\cdot)[r]D^{1-\alpha}_{q-}\omega_{q-}[r]\|_{L_{2,\otimes,\otimes}}\\
\leq  & \frac{\|\omega_S(r,t)\tilde E(u(r)-u(s),D^{1-\alpha}_{q-}\omega_{q-}[r])\|_{L_{2,\otimes,\otimes}}}
    {(r-s)^\alpha}\\
&+c \int_s^r\frac{\|(\omega_S(r,t)-\omega_S(\theta,t))\tilde E(u(r)-u(\theta),D^{1-\alpha}_{q-}\omega_{q-}[r])\|_{L_{2,\otimes,\otimes}}}{(r-\theta)^{1+\alpha}}
    d\theta \\
\leq &cc_{V,\hat V}(q-r)^{\alpha+\beta^\prime-1} s^{-\beta}\|\tilde E\|_{L_2(V\otimes V,\hat V)}\|u\|_{\beta,\sim}|||\omega|||_{\beta^\prime}\\
&\qquad \times (|||\omega|||_{\beta^\prime}+|||\omega|||_{\beta^{\prime\prime}}) (
(t-r)^{\beta^\prime}(r-s)^{\beta-\alpha}+
(r-s)^{\beta^\prime+\beta-\alpha}),
\end{split}
\end{align*}
which follows by Lemma \ref{locall24}. Hence, integrating the right hand side of the previous expression between $s$ and $q$ we get
\begin{align*}
 \|I_2(\cdot)\|_{L_{2,\otimes,\otimes}} & \leq c\|U\|_Ws^{-\beta}(t-s)^{\beta^\prime}(q-s)^{\beta+\beta^\prime},
\end{align*}
with $c$ depending on $|||\omega|||_{\beta^{\prime\prime}}$  and $|||\omega|||_{\beta^{\prime}}$. Now we estimate
\begin{align*}
I_1(\tilde E):=\int_s^qD_{s+}^{2\alpha-1}\tilde E(u(\cdot)-u(s),\cdot))[r]
D_{q-}^{1-\alpha}\dD_{q-}^{1-\alpha}(\omega_S(t)\otimes\omega)[r] dr.
\end{align*}
We emphasize that the expression $(\omega_S(t)\otimes\omega)$ is not well defined by an integral similar to \eqref{localeq1} for nonregular $\omega$. Nevertheless, splitting this expression as in Lemma \ref{uomom} we can express it in terms of $(\omega\otimes_S\omega)$
which is well defined by {\bf(H3)}. In addition, we can work with an approximation argument by the assumption {\bf(H3)}.\\
Now we split the previous integral into three integrals due to the definition (\ref{oS}). To treat the corresponding first expression let us write down the following estimate for $\alpha<\gamma<1,\,\beta^\prime<\gamma$:
\begin{align*}
  |D_{s+}^\alpha ((S(\tau-\cdot)&-{\rm id})(-A)^{-\beta^\prime}e)[r]| \le  c\bigg(\frac{|(S(\tau-r)-{\rm id})(-A)^{-\beta^\prime}e|}{(r-s)^{\alpha}}\\
  &+
\int_s^r\frac{|(S(\tau-r)-S(\tau-q))(-A)^{-\beta^\prime}e|}{(r-q)^{1+\alpha}}dq \bigg)\\
  \le & c\bigg(\frac{(\tau-r)^{\beta^\prime}|e|}{(r-s)^\alpha}+\int_s^r\frac{(|(S(r-q)-{\rm id})(-A)^{-\beta^\prime}S(\tau-r)e|}{(r-q)^{1+\alpha}}dq\bigg)\\
  \le &c\bigg(\frac{(\tau-r)^{\beta^\prime}}{(r-s)^\alpha}+\frac{(\tau-r)^{\beta^\prime-\gamma}}{(r-s)^{\alpha-\gamma}}\bigg)|e|,
\end{align*}
for $e\in V$, which follows by (\ref{eq1}) and (\ref{eq2}). Note we have
\begin{align*}
  \omega_S&(\tau,t) \int_s^\tau (S(\tau-r)-{\rm id})Ed\omega(r)=\omega_S(\tau,t)(-A)^{\beta^\prime}\int_s^\tau (S(\tau-r)-{\rm id})(-A)^{-\beta^\prime}E d\omega(r),
\end{align*}
hence, by the third statement of Lemma \ref{locall24} (i) and Lemma \ref{locall32} (ii), we conclude that
\begin{align*}
\| \omega_S&(\tau,t) \int_s^\tau (S(\tau-r)-{\rm id})\cdot d\omega(r)\|_{L_2,\otimes}\\
&\leq c |||\omega|||_{\beta^{\prime \prime} }\|\int_s^\tau (S(\tau-r)-{\rm id})(-A)^{-\beta^\prime}\cdot d\omega(r)\|_{L_2(L_2(V,\hat V),V)}\\
   &\le cc_{V,\hat V} |||\omega|||_{\beta^{\prime \prime} } \int_s^\tau|D_{s+}^\alpha((S(\tau-\cdot)-{\rm id})(-A)^{-\beta^\prime}[r]| |D_{\tau-}^{1-\alpha}\omega_{\tau-} [r]|dr\\
  &\le c |||\omega|||_{\beta^\prime} |||\omega|||_{\beta^{\prime \prime} } \int_s^\tau\bigg(\frac{(\tau-r)^{\beta^\prime}}{(r-s)^\alpha}+\frac{(\tau-r)^{\beta^\prime-\gamma}}{(r-s)^{\alpha-\gamma}}\bigg)(\tau-r)^{\beta^\prime+\alpha-1}dr\\
  &\le c |||\omega|||_{\beta^{\prime \prime} } |||\omega|||_{\beta^\prime} (\tau-s)^{2\beta^\prime}.
\end{align*}
For the other terms of the right-hand side of $(\omega_S \otimes \omega)$, by Hypothesis {\bf (H3)} and Lemma \ref{locall24} we have that
\begin{align*}
\|(-1)^{-\alpha} & \cdot(\omega\otimes_S\omega)(s,\tau)+(\omega_S(\tau,t)-\omega_S(s,t))\cdot(\omega(\tau)-\omega(s))\|_{L_{2,\otimes}}\le c   (\tau-s)^{2\beta^\prime},
\end{align*}
where $c$ depends on $\|(\omega\otimes_S\omega)\|_{2\beta^\prime}$ and $|||\omega|||_{\beta^{\prime \prime} }$. Since $\omega^n$ is smooth, the expression $\omega^n_S(t)\otimes \omega^n$ has the same structure as the integral \eqref{localeq1} and satisfies the Chen--equality. Moreover, we have the convergence of $(\omega^n_S(t)\otimes \omega^n)$ to $(\omega_S(t)\otimes\omega)$ in $L_{2,\otimes}$ such that the latter term satisfies the Chen--equality too. This convergence holds because all expressions in $(\omega_S\otimes \omega)$ depend linearly or  bilinearly on $\omega$ or
$(\omega\otimes_S\omega)$. Then the regularity of $(\omega_S(t)\otimes\omega)$
yields
\begin{align}\label{ea}
    &\|\cdot D_{q-}^{1-\alpha}\dD_{q-}^{1-\alpha}  (\omega_S(t)\otimes\omega)\|_{L_{2,\otimes}}\le
    c(q-r)^{2\beta^\prime+2\alpha-2}.
\end{align}
To establish the previous inequality we have to use that $(\omega_S\otimes \omega)$ is $2\beta^\prime$--H\"older continuous as well as the Chen--equality (in fact \eqref{ea} looks similar to the first property of Lemma \ref{locall3}, but it is easier to derive, and thus its complete proof is left to the reader). Finally, \eqref{ea} allows us to treat the  integral $I_1(\tilde E)$, obtaining a similar estimate to the ones we already have for $I_2$ and $I_3$ above.
\end{proof}

{\bf Proof of Lemma \ref{locall3}.}

\begin{proof}
We focus on proving the second estimate, since the first one is easier. Hence, we want to calculate $\|D_{t-}^{1-\alpha}\dD_{t-}^{1-\alpha}w(t,\cdot,\cdot)[r]\|_{L_{2,\otimes,\otimes}}$ for which we take into account the expression:
\begin{align}\label{localeq17}
\begin{split}
  & D_{t-}^{1-\alpha}\dD_{t-}^{1-\alpha}w(t,\cdot,\cdot)[r]=\frac{(-1)^\alpha}{\Gamma(\alpha)}\bigg(\frac{\dD_{t-}^{1-\alpha}w(t,\cdot,\cdot)[r]}{(t-r)^{1-\alpha}}\\
 &+(1-\alpha)\int_r^t\frac{\dD_{t-}^{1-\alpha}w(t,\cdot,\cdot)[r]-\dD_{t-}^{1-\alpha}w(t,\cdot,\cdot)[\theta]}{(\theta-r)^{2-\alpha}}d\theta\bigg)\\
 &=\frac{(-1)^\alpha}{\Gamma(\alpha)}\bigg(\frac{\dD_{t-}^{1-\alpha}w(t,\cdot,\cdot)[r]}{(t-r)^{1-\alpha}}\\
 &+(1-\alpha)\int_r^{t} \frac{ \frac{w(t,r,t)}{(t-r)^{1-\alpha}}-\frac{w(t,\theta ,t)}{(t-\theta )^{1-\alpha}}+(1-\alpha)\int_{r}^t\frac{w(t,r,\zeta)}{(\zeta-r)^{2-\alpha}}d\zeta-(1-\alpha)\int_{\theta}^t\frac{w(t,\theta,\zeta)}{(\zeta-\theta)^{2-\alpha}}d\zeta} {(\theta -r)^{2-\alpha}} d\theta \bigg).
\end{split}
\end{align}
We start by estimating the non-integral terms of the last expression. We obtain
\begin{align*}
  \bigg\|\frac{w(t,r,t)}{(t-r)^{1-\alpha}}&-\frac{w(t,\theta,t)}{(t-\theta)^{1-\alpha}}\bigg\|_{L_{2,\otimes,\otimes}}\le \frac{\|w(t,\theta,t)\|_{L_{2,\otimes,\otimes}}|(t-r)^{1-\alpha}-(t-\theta)^{1-\alpha}|}{(t-r)^{1-\alpha}(t-\theta)^{1-\alpha}}\\
  &+\frac{\|w(t,r,t)-w(t,\theta,t)\|_{L_{2,\otimes,\otimes}}}{(t-r)^{1-\alpha}} \\
  \le& c r^{-\beta}\|U\|_W\big( (t-\theta)^{\beta^\prime+\beta}(t-\theta)^{\beta^\prime}(\theta-r)^\beta(t-\theta)^{1-\alpha-\beta}(t-r)^{\alpha-1}(t-\theta)^{\alpha-1}\\
  &+(t-r)^{\beta^\prime}(\theta-r)^{\beta^\prime+\beta}(t-r)^{\alpha-1}+(\theta-r)^\beta(t-\theta)^{2\beta^\prime}(t-r)^{\alpha-1}\big).
\end{align*}
To get these estimates we have used Lemma \ref{locall9} and in addition, for the first expression on the right hand side we have used the trivial inequality
\begin{align}\label{localtr}
y^{1-\alpha}-x^{1-\alpha}\leq (y-x)^\beta x^{1-\alpha-\beta}
\end{align}
for any $x<y$, given in Lemma 6.1 of \cite{HuNu09}, while we have managed the second one by using the Chen--equality \eqref{localeq10b}.
Therefore, we obtain
\begin{align*}
  \bigg\|\frac{w(t,r,t)}{(t-r)^{1-\alpha}}&-\frac{w(t,\theta,t)}{(t-\theta)^{1-\alpha}}\bigg\|_{L_{2,\otimes,\otimes}}\le c r^{-\beta}\|U\|_W(\theta-r)^\beta(t-r)^{2\beta^\prime+\alpha-1}.
\end{align*}
In addition, for the integral terms of the last expression of (\ref{localeq17}) we have
\begin{align*}
  &\bigg\|\int_{r}^t\frac{w(t,r,\zeta)}{(\zeta-r)^{2-\alpha}}d\zeta-\int_{\theta}^t\frac{w(t,\theta,\zeta)}{(\zeta-\theta)^{2-\alpha}}d\zeta\bigg\|_{L_{2,\otimes,\otimes}} \\
  &\le \int_{r}^{\theta}\frac{\|w(t,r,\zeta)\|_{L_{2,\otimes,\otimes}} }{(\zeta-r)^{2-\alpha}}d\zeta
  +\bigg(\int_{\theta}^t\bigg\|\frac{w(t,r,\zeta)}{(\zeta-r)^{2-\alpha}}-\frac{w(t,\theta,\zeta)}{(\zeta-\theta)^{2-\alpha}}\bigg\|_{L_{2,\otimes,\otimes}}  d\zeta\bigg)=:B_1+B_2.
\end{align*}
On account of Lemma \ref{locall19b}, due to $\beta+\beta^\prime+\alpha>1$, for $B_1$ we have
\begin{equation*}
  B_1\le  \frac{c}{r^{\beta}}\|U\|_{W}\int_{r}^{\theta}\frac{(\zeta-r)^{\beta+\beta^\prime}(t-r)^{\beta^\prime}}{(\zeta-r)^{2-\alpha}}d\zeta\le \frac{c}{r^\beta}\|U\|_W(t-r)^{\beta^\prime}(\theta-r)^{\beta+\beta^\prime+\alpha-1},
\end{equation*}
and for $B_2$
\begin{align*}
  B_2\le & \int_{\theta}^{t}\frac{\|w(t,r,\zeta)-w(t,\theta,\zeta)\|_{L_{2,\otimes,\otimes}}}{(\zeta-r)^{2-\alpha}}d\zeta \\
  + & \int_{\theta}^t \frac{\|w(t,\theta,\zeta)\|_{L_{2,\otimes,\otimes}}|(\zeta-\theta)^{2-\alpha}-(\zeta-r)^{2-\alpha}|}{(\zeta-r)^{2-\alpha}(\zeta-\theta)^{2-\alpha}}d\zeta
  =:B_{21}+B_{22}.
\end{align*}
For $B_{22}$, using Lemma \ref{locall19b} and the estimate \eqref{localtr} we get:
\begin{align*}
  B_{22}\le & \frac{c}{r^{\beta}}\|U\|_W (t-\theta)^{\beta^\prime} (\theta-r)^\beta \int_{\theta}^t(\zeta-r)^{\alpha-2} (\zeta-\theta)^{\beta^\prime}d\zeta\\
  \le &  c r^{-\beta}\|U\|_W(\theta-r)^\beta(t-r)^{2\beta^\prime+\alpha-1}.
\end{align*}
To estimate $B_{21}$ we need the generalized Chen--equality  (\ref{localeq9b}), giving us
\begin{align*}
  B_{21}\le& \frac{c}{r^\beta}\|U\|_W \int_{\theta}^t\bigg(\frac{(\theta-r)^{\beta+\beta^\prime}(t-r)^{\beta^\prime}}{(\zeta-r)^{2-\alpha}}
  +\frac{(\theta-r)^\beta(\zeta-\theta)^{2\beta^\prime}}{(\zeta-r)^{2-\alpha}}\\
  &+\frac{(\theta-r)^\beta(t-\zeta)^{\beta^\prime}(\zeta-\theta)^{\beta^\prime}}{(\zeta-r)^{2-\alpha}}\bigg)d\zeta\\
  \le& c r^{-\beta}\|U\|_W\bigg( (\theta-r)^{\beta+\beta^\prime-1+\alpha}(t-r)^{\beta^\prime}+(\theta-r)^\beta(t-\theta)^{2\beta^\prime+\alpha-1}\\
  &+(t-\theta)^{\beta^\prime}(t-\theta)^{\beta^\prime+\alpha-1}(\theta-r)^\beta\bigg)\\
  \le & c r^{-\beta}\|U\|_W(\theta-r)^\beta(t-r)^{2\beta^\prime+\alpha-1}.
\end{align*}
Indeed, taking into account (\ref{localeq9b}) we have to estimate in particular the expression
\begin{equation*}
\|\omega_S(\zeta,t)S_\omega(\zeta,\theta)\cdot(u(\theta)-u(r),\cdot)\|\le \frac{c}{r^\beta}|||\omega|||_{\beta^{\prime\prime}}^2(t-\zeta)^{\beta^\prime}(\zeta-\theta)^{\beta^\prime}(\theta-r)^\beta \|U\|_W
\end{equation*}
which can be easily done by Lemma \ref{locall24} (i) and (ii). Therefore, the last expression of (\ref{localeq17}) gives us as estimate
\begin{align*}
c r^{-\beta}(t-r)^{2\beta^\prime +\alpha-1}\|U\|_W \int_r^t \frac{(\theta-r)^\beta}{(\theta-r)^{2-\alpha}}d\theta \leq c r^{-\beta}(t-r)^{2\beta^\prime+\beta +2\alpha-2} .
\end{align*}

Similar we obtain an estimate for the first expression on the right side of \eqref{localeq17} setting $\theta=t$. \\
Finally, note that the appearing constant $c$ depends on $|||\omega|||_{\beta^{\prime\prime}}$ and $\|(\omega\otimes_S\omega)\|_{2\beta^\prime}$.
\end{proof}

\begin{corollary}\label{coro3}
The proof of \eqref{localeq21} follows the same steps than the last proof, with the difference that we do not have to use \eqref{localeq9} but the Chen--equality (\ref{chen}).
\end{corollary}

\end{document}